\documentclass[12pt,a4paper,abstracton]{scrartcl}

\usepackage[utf8]{inputenc}
\usepackage[T1]{fontenc}
\usepackage[ngerman, english]{babel}
\usepackage{braket}				
\usepackage{color}
\usepackage[nottoc]{tocbibind}	

\usepackage{appendix}

\usepackage{url}

\usepackage{tikz-cd}
\usetikzlibrary{babel}

\usepackage{textgreek}				
\usepackage{hyperref}

\usepackage{mathtools}				

\usepackage{enumitem}

\usepackage{mathdots}

\usepackage{dsfont} 

\usepackage{subcaption}					

\usepackage{latexsym}
\usepackage{amsmath,amssymb,amsthm}
\usepackage{colonequals}							

\usepackage{faktor}

\usepackage[onehalfspacing]{setspace}			

\makeatletter																
\def\thm@space@setup{%
  \thm@preskip=7mm plus 2mm minus 1mm
  \thm@postskip=\thm@preskip 
}
\makeatother

\newtheorem{Theorem}{Theorem}[section]
\newtheorem{Lemma}[Theorem]{Lemma}		   
\newtheorem{Corollary}[Theorem]{Corollary}

\theoremstyle{definition}
\newtheorem{Definition}[Theorem]{Definition}
\newtheorem{Example}[Theorem]{Example}
                  
\numberwithin{equation}{section} 

\newcommand{\R}{\mathbb{R}} 
\newcommand{\Z}{\mathbb{Z}} 
\newcommand{\N}{\mathbb{N}} 

\newcommand{\bracket}[1]{\, \lbrack \, #1 \, \rbrack \,}				

\newenvironment{acknowledgements}{\begin{abstract}} {\end{abstract}}   

\mathchardef\ordinarycolon\mathcode`\:						
\mathcode`\:=\string"8000
\begingroup \catcode`\:=\active
	\gdef:{\mathrel{\mathop\ordinarycolon}}
\endgroup

\newcommand{\meet}{\mathrel{\text{{\scalebox{0.8}{$\wedge$}}}}}			
\newcommand{\join}{\mathrel{\text{{\scalebox{0.8}{$\vee$}}}}}				

\DeclareMathOperator{\rk}{rk}																	
\DeclareMathOperator{\st}{st}
\DeclareMathOperator{\lk}{lk}
\DeclareMathOperator{\cl}{cl}
\DeclareMathOperator{\Cone}{C}
\DeclareMathOperator{\NC}{NC}
\DeclareMathOperator{\K}{K}
\DeclareMathOperator{\colim}{colim}

\makeatletter																														
\DeclareRobustCommand*{\mfaktor}[3][]
{
   { \mathpalette{\mfaktor@impl@}{{#1}{#2}{#3}} }
}
\newcommand*{\mfaktor@impl@}[2]{\mfaktor@impl#1#2}
\newcommand*{\mfaktor@impl}[4]{
   \settoheight{\faktor@zaehlerhoehe}{\ensuremath{#1#2{#3}}}%
   \settoheight{\faktor@nennerhoehe}{\ensuremath{#1#2{#4}}}%
      \raisebox{-0.5\faktor@zaehlerhoehe}{\ensuremath{#1#2{#3}}}%
      \mkern-4mu\diagdown\mkern-5mu%
      \raisebox{0.5\faktor@nennerhoehe}{\ensuremath{#1#2{#4}}}%
}
\makeatother

\begin{document}
  \pagestyle{empty}
  \graphicspath{{Figures/}}

  \begin{titlepage}
{\setstretch{1.0}
    \includegraphics[scale=0.45]{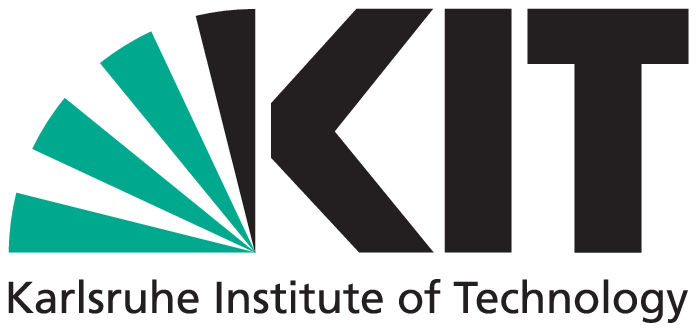}
    \vspace*{2cm} 

 \begin{center} \large 
    
    Master's Thesis
    \vspace*{2cm}

    {\huge On Brady's Classifying Spaces\\
    \vspace*{3mm}for Artin Groups of Finite Type}
    \vspace*{2.5cm}

    Valentin Braun
    \vspace*{1.5cm}

    March 6, 2018
    \vspace*{4.5cm}

    Advisor: JProf. Dr. Petra Schwer\\[1cm]
    Department of Mathematics \\[1cm]
		Karlsruhe Institute of Technology
  \end{center}
  
  }
\end{titlepage}

\pagestyle{plain}
\setcounter{page}1
\pagenumbering{Roman}

{\setstretch{1.7}
\begin{acknowledgements}
I am very grateful to Julia Heller for many hours of interesting discussions and valuable advice I got on my way to understanding the topic and writing the thesis.
She has always been a patient listener to my concerns and provided me with guidance in a kind and dependable way.
I greatly appreciate her help and support I received during the process.
\end{acknowledgements}
}

\newpage

\begin{minipage}{\linewidth}
\begin{abstract}
This thesis takes Brady's construction of $K(\pi,1)$s for the braid groups as a starting point. It is widely known that this construction can – with the right ingredients – be generalized to Artin groups of finite type. Results of Bessis as well as Brady and Watt are used to establish the general construction for Artin groups of finite type. The non-crossing partition lattice in finite Coxeter groups is identified and used to generate the so called poset group. With the help of this poset group, which turns out to be isomorphic to the Artin group, a simplicial complex is constructed on which the poset group acts. It is shown that the complex itself is the universal cover of the $K(\pi,1)$ of the Artin group of finite type and the quotient under the action is the desired $K(\pi,1)$.
\end{abstract}
\vspace*{4cm}
\selectlanguage{ngerman}
\begin{abstract}

Deutscher Titel: \textbf{Über Bradys klassifizierende Räume für Artin-Gruppen endlichen Typs}

Diese Thesis nimmt Bradys Konstruktion von $K(\pi,1)$-Räumen für die Zopfgruppen als Ausgangspunkt. Es ist allgemein bekannt, dass diese Konstruktion mit den richtigen Hilfsmitteln auf die Klasse der Artin-Gruppen endlichen Typs verallgemeinert werden kann. Es werden Ergebnisse von Bessis sowie Brady und Watt benutzt um die Verallgemeinerung auf Artin-Gruppen endlichen Typs möglich zu machen. Nichtkreuzende Partitionen in endlichen Coxetergruppen werden identifiziert und benutzt um die sogenannte Poset-Gruppe zu erzeugen. Diese ist isomorph zur Artin-Gruppe und ist Basis für die Konstruktion eines Simplizialkomplexes auf dem die Gruppe wirkt. Es wird gezeigt, dass dieser Simplizialkomplex die universelle Überlagerung des $K(\pi,1)$ der Artin-Gruppe endlichen Typs ist und dass der Quotientenraum der Gruppenwirkung der gewünschte $K(\pi,1)$ ist.
\end{abstract}
\end{minipage}

\newpage
  \tableofcontents
  
  \newpage
  
  \listoffigures

\newpage

  \pagestyle{headings}

\setcounter{page}1
\pagenumbering{arabic}

\section{Introduction}

The class of Artin groups arises from a generalization of the braid groups, which were introduced by Artin in \cite{ART25}. In this work, Artin proved a certain finite presentation for the braid groups and solved the word problem. Other fundamental work on braid groups includes Garside's \cite{GAR69}, where a new solution to the word problem is given and the conjugacy problem is solved. Brieskorn and Saito \cite{BS72} and Deligne \cite{DEL72} independently generalized the braid groups to what is known today as Artin groups.\\

The aim of this thesis is to describe Brady's construction of $\K(\pi,1)$s for the braid groups from \cite{BRA01} and elaborate on the general case of Artin groups of finite type.

\begin{Definition}
	Let $\pi$ be a group and $X$ a connected topological space. Then $X$ is called an \emph{Eilenberg-MacLane space of type }$\K(\pi,1)$, if
	\begin{alignat*}{2}
		&\pi_1(X) \cong \pi \qquad &\text{ and} \\
		&\pi_n(X) \text{ is trivial} \qquad &\text{ for } n \geq 2.
	\end{alignat*}
\end{Definition}

Instead of Eilenberg-MacLane space of type $\K(\pi,1)$ we will also call such spaces just $\K(\pi,1)$s. In case $\pi$ is a discrete group, $\K(\pi,1)$ spaces are also called \emph{classifying spaces}.\\

Some of the results of \cite{BRA01} are closely related to Birman, Ko, Lee \cite{BKL98}. Krammer described independently from Brady in \cite{KRA00} the construction of the very same $\K(\pi,1)$ for the braid groups. That is why in the literature this complex is sometimes called \emph{Brady-Krammer complex}.

In \cite{BW02} Brady and Watt generalized this construction to Artin groups of type $C_n$ and $D_n$. In the same work they noted that the construction can be generalized to every Artin group of finite type for which one can show that the closed interval $[I, \gamma]$ forms a lattice in $W$ equipped with the reflection order, where $I$ is the identity in the related Coxeter group $W$ and $\gamma$ a Coxeter element. Bessis independently obtained similar results and already established this lattice property for all finite Coxeter groups with a case-by-case proof, that was partly achieved by computer, in \cite{BES03}. Brady and Watt then gave a case-free proof of the lattice property for all finite Coxeter groups in \cite{BW08}.

The main goal of \cite{BES03} was the study of the \emph{dual braid monoid} of Artin groups of finite type. This approach can be understood as a dual theory of the \emph{positive braid monoid}, which was for example studied in \cite{GAR69}. Instead of the positive braid monoid, which comes from the Coxeter group with its standard generating set $(W,S)$, the monoid which is generated by the set of all reflections $T= \set{wsw^{-1} | w \in W, s \in S }$ is considered. Bessis introduced this notion of \emph{dual Coxeter theory}, of which the construction in this thesis also makes use, in \cite{BES03}. The case when $(W,S)$ is of type $A_n$ was considered earlier in \cite{BKL98}.\\

The thesis is structured as follows.

In Section~\ref{sec:def} basic definitions are made. Then, in Section~\ref{sec:lat}, the reflection order on a Coxeter group is defined and we observe the non-crossing partition lattice described by Brady and Watt in \cite{BW08}. We also prove a few lemmas on the structure of the lattice. We then define the poset group for a finite Coxeter group in Section~\ref{Sec:PosetGr} and use a result of Bessis \cite{BES03} to establish an isomorphism to the related Artin group. We establish cancellation properties in the positive semi group and show that it embeds into the poset group. In Section~\ref{sec:complex} we construct a simplicial complex and show that it is the universal cover of the desired $K(\pi,1)$.\\

We describe the combinatorial construction of the Brady-Krammer complex, closely following \cite{BRA01} and give the general approach for Artin groups of finite type. We use the main result of \cite{BW08} to establish the lattice property and make use of a result of Bessis \cite{BES03}, where he showed that the poset group of a finite Coxeter group – defined in Section~\ref{Sec:PosetGr} – is isomorphic to its related Artin group. We try to elaborate, give explaining examples and go into detail in the proofs so that the construction becomes understandable and easy to read.

\newpage

\section{Definitions and Notions}\label{sec:def}

For the most part, we will follow the notation of Björner and Brenti \cite{BB05}.
\subsection{Partially Ordered Sets}

\begin{Definition}
	A \emph{partially ordered set} – or short \emph{poset} – is a pair $(P,\leq)$ of a set $P$ and a relation $\leq$ on $P$ with the following properties. For all $p,q,r \in P$ it holds:
	\begin{enumerate}[label=(\roman*)]
		\item $p \leq p$ \quad (reflexivity),
		\item if $p \leq q$ and $q \leq r$, then $p \leq r$ \quad (transitivity),
		\item if $p \leq q$ and $q \leq p$, then $p = q$ \quad (anti-symmetry).
	\end{enumerate}
\end{Definition}

If the order relation $\leq$ is clear from the context, we call $P$ a poset. If $p \leq q$ but $p\neq q$, we also write $p<q$. A \emph{cover relation} is a pair $p < q$ such that there is no $r \in P$ with $p < r < q$.

For $Q \subseteq P$ we call $(Q, \leq)$ a \emph{subposet} of $(P, \leq)$ when $Q$ inherits the order of $P$. For $p,q \in P$ we define the interval $\bracket{p,q} := \set{r \in P | p \leq r \leq q}$.
A sequence $(p_0, p_1, \dots, p_k)$ of elements of $P$ is called a \emph{chain} if $p_0< p_1< \dots < p_k$, where $k$ is called the \emph{length} of the chain. The supremum of the lengths of all chains of  $P$ is called the \emph{rank} of $P$. A chain is \emph{maximal} if its elements are not a proper subset of the elements of any other chain. We call $P$ \emph{pure} if all maximal chains are of the same finite length. We call an element $p \in P$ \emph{maximal} if there is no $q \in P$ with $p < q$.

For a pure poset $P$ we define a \emph{rank} function $\rk\colon P \to \N_0$ by letting, for $p \in P$, $\rk(p)$ be the rank of the subposet $\set{q \in P | q \leq p}$.

\vspace*{17pt}
Let $P$ be a poset and $p,q,u,l \in P$.
We call $l$ the \emph{meet} of $p$ and $q$ and write $l = p \meet q$ if
\begin{itemize}
	\item $l \leq p$ and $l \leq q$ and
	\item for all $r \in P$ with $r \leq p $ and $r \leq q$ it holds $r \leq l$.
\end{itemize}
Analogously, we call $u$ the \emph{join} of $p$ and $q$ and write $u = p \join q$ if
\begin{itemize}
	\item $p \leq u$ and $q \leq u$ and
	\item for all $r \in P$ with $p \leq r$ and $q \leq r$ it holds $u \leq r$.
\end{itemize}

The meet is the greatest lower bound, while the join is the least upper bound. Meet and join are – if they exist – necessarily unique. If in a poset for all pairs of elements there exists a meet and a join, we call it a \emph{lattice}.

The \emph{Hasse diagram} of a poset $(P,\leq)$ is a diagram that represents its structure. It is a graph on the vertex set $P$ and for each cover relation $p \leq q$ there is an edge going upwards from $p$ to $q$.

\begin{Example}\label{ex:poset1}

	Let $P= \{1,2,3,4\}$ with partial order relation \[ 1 \leq 1,\quad 2 \leq 2,\quad 3 \leq 3,\quad 4 \leq 4,\quad 2 \leq 1,\quad 2 \leq 3,\quad 2 \leq 4,\quad 1 \leq 3 . \]
	
	The Hasse diagram of $(P, \leq)$ is depicted in Figure~\ref{fig:Hasse_ex}.
	\vspace*{9mm}
	\begin{figure}[h]
		\centering
		\includegraphics[width= 0.15\textwidth]{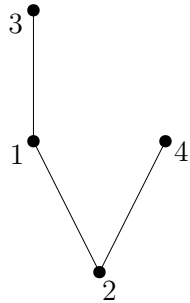}
		\caption{Hasse diagram of $(P,\leq)$}
		\label{fig:Hasse_ex}
	\end{figure}
\end{Example}

\subsection{Simplicial Complexes}\label{subsec:SK}

\begin{Definition}
	A non-empty family $\Delta$ of finite subsets of a set $V$ is called \emph{abstract simplicial complex}, if for any $F \in \Delta$ and $F' \subseteq F$ it also holds $F' \in \Delta$.
\end{Definition}

The set $V$ is called the \emph{set of vertices}. The elements of $\Delta$, which are subsets of $V$, are called \emph{faces}. Note that since we require an abstract simplicial complex to be non-empty, we always have $\emptyset \in \Delta$. Furthermore, we only consider abstract simplicial complexes where for each vertex $v$, $\{v\} \in \Delta$ is a face. In this case we identify $\{v\}$ with the vertex $v \in V$.

The \emph{dimension} of a face $F$ is defined as $\dim(F) = |F| - 1$. The dimension of $\Delta$ is defined as the supremum of the dimensions of all faces of $\Delta$. For $F' \subseteq F$, we define the interval $\bracket{F', F} := \set{G \in \Delta|F' \subseteq G \subseteq F}$.

A subset of $\Delta$ which is an abstract simplicial complex itself is called a \emph{subcomplex} of $\Delta$. The specific subcomplex 
$\bracket{\emptyset,F}$ for a face $F \in \Delta$ is called a \emph{simplex}.

For $n \in \N$, we define the \emph{$n$-skeleton} of $\Delta$ to be the subcomplex consisting of all faces of dimension at most $n$. Faces of dimension $1$ are also called \emph{edges}.

If $P$ is a poset, then we can associate to $P$ a specific complex $\Delta(P)$.

\begin{Definition}
	For a poset $P$, let $\Delta(P)$ be the following set.
	\[
		\Delta(P) := \big\{ \{p_0, p_1, \dots, p_k \} \subseteq P \, \big| \, p_0 < p_1 < \dots < p_k \text{ is a finite chain in } P \big\}.
	\]
	
\end{Definition}

The elements of $\Delta(P)$ are just all finite chains of $P$. Since all subsets of finite chains are finite chains themselves, $\Delta(P)$ is closed under containment and therefore $\Delta(P)$ is an abstract simplicial complex.

The complex $\Delta(P)$ is called the \emph{order complex} of $P$.

\begin{Definition}\label{def:SK}

	Let $\Delta$ be an abstract simplicial complex, $v$ a vertex of $\Delta$, $F \in \Delta$ a face and $S$ a collection of faces of $\Delta$.
	
	\begin{enumerate}[label=(\roman*)]

		\item The \emph{closure} of $S$, denoted as $\cl(S,\Delta)$, is the smallest subcomplex of $\Delta$ that contains all faces of $S$, i.e.
		\[
			\cl(S, \Delta) = \set{F \in \Delta | \exists F' \in S \colon F \subseteq F' }.
		\]
		For a face $F$, the closure $\cl(F,\Delta)$ is defined as $\cl\big( \{ F \}, \Delta \big)$.
			
		\item The \emph{star} of $v$, denoted as $\st(v,\Delta)$, is the smallest simplicial complex that contains all faces of $\Delta$ which contain $v$ as a vertex, i.e.
		\[
			\st(v,\Delta) := \cl\big( \set{F \in \Delta|v \in F}, \Delta \big).
		\]
		
		\item The \emph{link} of $v$, denoted as $\lk(v, \Delta)$, is the set that contains all faces of $\st(v,\Delta)$ which do not contain $v$, i.e.
		\[
			\lk(v, \Delta) := \big\{ F \in \st(v,\Delta) \, \big| \, v \notin F \big\}.
		\]
		
		\item The \emph{(simplicial) cone} over $\Delta$, denoted as $\Cone(\Delta)$, is obtained by introducing a new vertex $c$ to $\Delta$ and adding for each face
		$F$ of $\Delta$, the face $F' := \{c\} \cup F$.
		
		To specify the \emph{cone vertex} $c$ of a cone over the complex $\Delta$, we sometimes write $\Cone_c(\Delta)$.
	
	\end{enumerate}

\end{Definition}

Note that unlike other authors, we define the star as a closed (under containment) subset. Thus, it is always an abstract simplicial complex itself. The link of a vertex is also always an abstract simplicial complex.

\begin{Example}\label{ex:SK_1}

	Let $\Delta := \big\{ \emptyset, \{1\}, \{2\}, \{3\}, \{4\}, \{1,2\}, \{1,3\}, \{2,3\}, \{2,4\},\{1,2,3\} \big\}$. Then $\Delta$ is an abstract simplicial complex. In fact, it is the order complex of the poset $(P, \leq)$ from Example~\ref{ex:poset1}.
	\begin{itemize}
		\item The closure of $\big\{ \{1\}, \{2,3\} \big\}$ is
		$\cl \left( \big\{ \{1\}, \{2,3\} \big\}, \Delta \right) = \big\{ \emptyset, \{ 1 \}, \{ 2 \}, \{ 3 \}, \{2, 3\} \big\}$.
		
		\item The star of $2$ is $\st(2, \Delta) = \Delta$.
	
		\item The link of $2$ is $\lk(2, \Delta) = \big\{ \emptyset, \{1\}, \{3\}, \{4\}, \{1,3\} \big\}$.
		
		\item The cone over $ L:= \big\{\emptyset, \{1\}, \{2\}, \{1,2\} \big\}$ with cone vertex $3$ is\\
		$\Cone_3(L) = \big\{ \emptyset, \{1\}, \{2\}, \{3\}, \{1,2\}, \{1,3\}, \{2,3\}, \{1,2,3\} \big\}$
	
	\end{itemize}

\end{Example}

The following lemma will be needed in the proof of Theorem~\ref{thm:Xplus}, but since it holds for any abstract simplicial complex, we will prove it now.

\begin{Lemma}\label{lem:cone}
	For an abstract simplicial complex $\Delta$ and a vertex $v$ of $\Delta$, it holds
	\[\st(v,\Delta) = \Cone_v \big( \lk(v,\Delta) \big).\]
\end{Lemma}

\begin{proof}
	Let $F$ be a face of $\st(v, \Delta)$. Then there are two cases.
	
	Either $F$ contains $v$ as a vertex. Then, $F' := F \setminus \{v\}$ is a face of $\st(v, \Delta)$, but does not contain $v$. Hence, we have $F' \in \lk(v, \Delta)$. But then, by definition of the cone, $F' \cup \{v\} = F$ must be a face of $\Cone_v \big( \lk(v,\Delta) \big)$.
	
	In the second case $F$ does not contain $v$. Thus, we have by definition of the link $F \in \lk(v, \Delta)$. But then $F$ is also a face of the cone over $\lk(v, \Delta)$.

\vspace{5mm}

	Now let $F$ be a face of $\Cone_v \big( \lk(v,\Delta) \big)$. Again, there are two cases.
	
	Either $v$ is a vertex of $F$, in which case $F$ is a face of $\st(v,\Delta)$.
	
	Or $v$ is not a vertex of $F$. Then, $F$ is already a face of $\lk(v, \Delta)$. But then, since the link is a subset of the star, we have $F \in \st(v, \Delta)$.
	
\end{proof}

For any abstract simplicial complex there is a topological space related to it. Such spaces are called \emph{geometric simplicial complexes} and are the geometric counterpart to abstract simplicial complexes.

A \emph{geometric simplex} of dimension $k$ is the convex hull of $k+1$ affinely independent points in $\R^n$ for $n \geq k$. The affinely independent points are called the \emph{vertices} of the geometric simplex and the convex hull of any subset of the vertices is called a \emph{face} of the simplex.

\begin{Definition}
	
	A \emph{geometric simplicial complex} $K$ is a non-empty collection of geometric simplices, such that
	\begin{itemize}
		\item any face of a simplex of $K$ is a simplex of $K$ and
		\item two simplices of $K$ intersect in a common face.
	\end{itemize}
	
\end{Definition}

Note that since we defined a face of a simplex to be the convex hull of \emph{any} subset of its vertices, the second condition could also mean that the intersection of two simplices is empty.

From any geometric simplicial complex $K$ we can derive an abstract simplicial complex by taking the sets of vertices of the geometric simplices of $K$ to be the faces of the abstract complex. In fact, any abstract simplicial complex can be obtained this way.\\

The other way around is more interesting for us; from any abstract simplicial complex $\Delta$ we can derive a geometric simplicial complex, denoted as $\|\Delta\|$, called its \emph{geometric realization}. This geometric realization is not unique, but all geometric realizations – regarded as topological spaces – are the same up to homeomorphism. Therefore we will talk about \emph{the} geometric realization. It is fully determined by the combinatorial properties of the abstract simplicial complex.\\

A sketch of one way to construct the geometric realization of a given abstract simplicial complex $\Delta$ is the following.

Take for every face $F \in \Delta$ one geometric simplex $s_F$ of dimension $\dim(F)$, and define a map $f_F \colon F \to s_F$ which identifies the vertices of the face $F$ of the abstract complex with the vertices of the geometric simplex $s_F$. Then, define for each pair $F_1, F_2 \in \Delta$ with $F_1 \subseteq F_2$ the inclusion $\iota_{F_1F_2} \colon s_{F_1} \hookrightarrow s_{F_2}$ of the corresponding geometric simplices, such that the following diagram commutes.\\

\begin{center}
	\begin{tikzcd}
		F_1 \arrow[r, "f_{F_1}"] \arrow[d, "\rotatebox{270}{$\subseteq$}"']
		& s_{F_1} \arrow[d, hook, "\iota_{F_1F_2}"] \\
		F_2 \arrow[r, "f_{F_2}"]
		& s_{F_2}
	\end{tikzcd}\\
\end{center}

Let $\sim$ be the coarsest equivalence relation on $\bigsqcup_{F \in \Delta} s_F$ such that for all
$x \in \bigsqcup_{F \in \Delta} s_F$ it holds $x \sim \iota_{F_1F_2} (x)$ for $F_1, F_2 \in \Delta$ with $F_1 \subseteq F_2$.

Then,
\[
	\| \Delta \| := \faktor{\bigsqcup_{F \in \Delta} s_F}{\sim}
\]
is the geometric realization of $\Delta$.

A more detailed explanation of this construction can be found in \cite{BOA02}, which is a note on §2.1 of Hatcher's \cite{HAT02}.\\

Other approaches to defining the geometric realization of an abstract simplicial complex can for example be found in Chapter~1, §4 of \cite{AS60}.

\begin{Example}

	Some geometric realizations of subsets of the abstract simplicial complex~$\Delta$ from Example~\ref{ex:SK_1} are the following.
	
	\begin{figure}[h!]
		\centering
		\begin{subfigure}{0.45\linewidth}
			\centering
     		\includegraphics[scale=0.53]{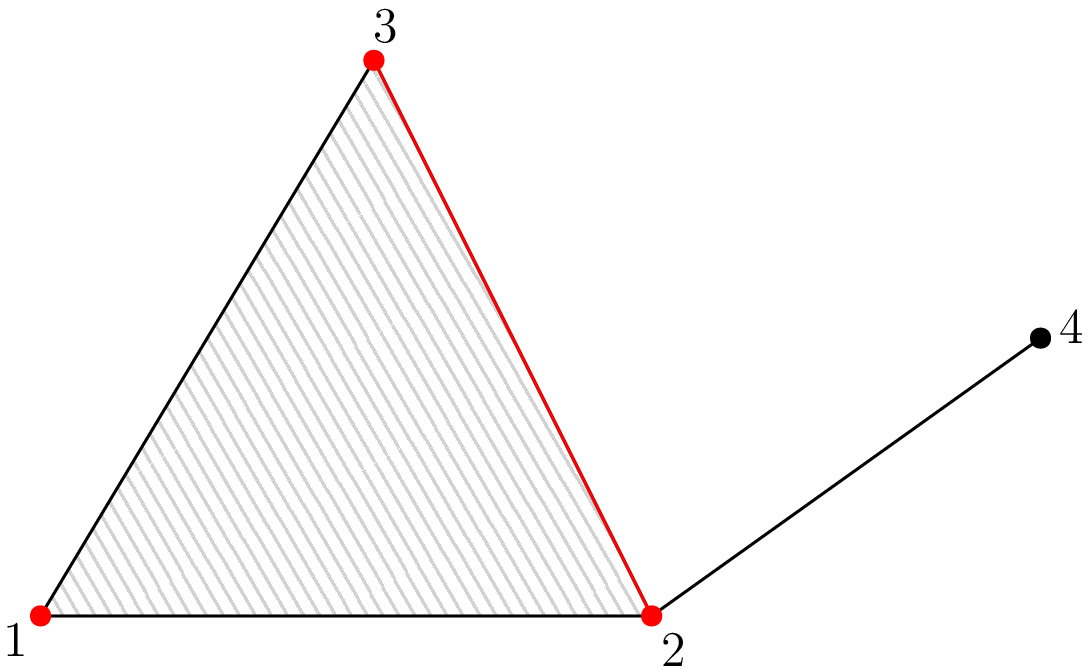}
     		\caption{The subcomplex $\cl \left( \big\{ \{1\}, \{2,3\} \big\}, \Delta \right)$ is highlighted.}\label{fig:SK_Ex1}
     	\end{subfigure}
     	\hspace*{10mm}
     	\begin{subfigure}{0.45\linewidth}
     		\centering
     		\includegraphics[scale=0.53]{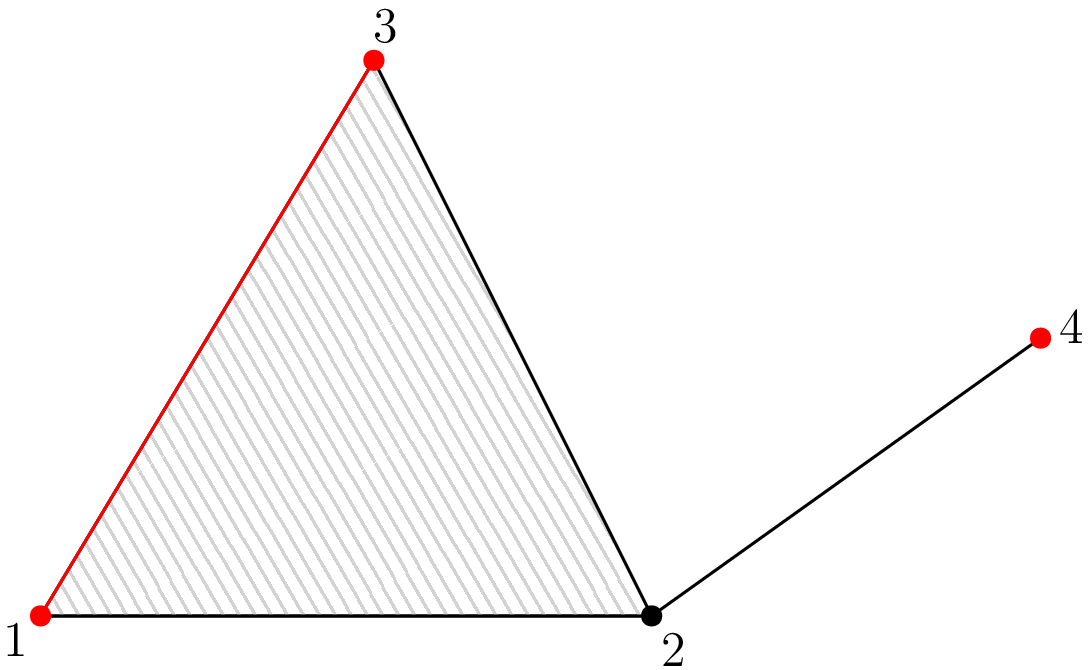}
     		\caption{The subcomplex $\lk(2, \Delta )$ is highlighted.}\label{fig:SK_Ex3}
     	\end{subfigure}
     	\caption{Geometric realizations of a simplicial complex} \label{fig:SK_Ex}
	\end{figure}

\end{Example}

Throughout this thesis we will often talk about `the simplicial complex' without specifying on whether we mean the abstract simplicial complex or its geometric realization. It should be clear that whenever we refer to combinatorial properties, we mean the abstract complex and when we refer to topological properties, such as contractibility, we mean its geometric realization.


\subsection{Coxeter Groups and Artin Groups}

A matrix $(m_{i,j})_{i,j=1, \dots, n}$, $n \in \N$, with $m_{i,j} \in \N \cup \{\infty\}$ is called a \emph{Coxeter matrix} if it is symmetric and $m_{i,j} = 1$ if and only if $i=j$.

\begin{Definition}
	A group $W$ is called a \emph{Coxeter group} if it admits a presentation $W=\langle S | R \rangle$ with generating set $S=\{s_1, \dots , s_n \}$ and defining relations $R$, consisting of
	\[
		s_i^2 = I \text{ for } i=1, \dots n, \quad
		\text{ and }  \underbrace{s_i s_j s_i\dots}_{m_{i,j}\text{ times}} =  \underbrace{s_j s_i s_j\dots}_{m_{i,j}\text{ times}} \text{ for all } i\neq j \text{ with } m_{i,j} \neq \infty
	\]
	of an $n\times n$ Coxeter matrix $(m_{i,j})$.
	
\end{Definition}

$(W,S)$ is called a \emph{Coxeter system} with Coxeter group $W$ and \emph{Coxeter generators} $S$. The cardinality of $S$ is called the \emph{rank} of $W$, denoted as $\rk(W)$. For a Coxeter group $W$, the generating set $S$ is not unique. Thus, we have to specify one set of Coxeter generators. The $1$s on the main diagonal of a Coxeter matrix just mean that the generators are all involutions. If an entry of the Coxeter matrix is $m_{i,j} = \infty$, this means that the product $(s_is_j)$ has `order infinity' – meaning $(s_is_j)^k \neq I$ for all $k \in \N$, where $I$ denotes the identity in $W$.\\

Another way of describing the relations of a Coxeter group is by the \emph{Coxeter graph}. The vertex set of the Coxeter graph of $(W,S)$ is $S$ and there is an edge joining $s_i$ and $s_j$ if $m_{i,j} \geq 3$. If $m_{i,j} \geq 4$, the edge joining $s_i$ and $s_j$ is labeled by $m_{i,j}$.

Note that we could also write the relations of a Coxeter group as $R=\{(s_i s_j)^{m_{i,j}} = I\}$ which includes the $s_i^2 = I$. But because of the following definition we prefer the former presentation.

\begin{Definition}
	Let $(W,S)$ be a Coxeter system with Coxeter matrix $(m_{i,j})_{i,j=1, \dots, n}$. The \emph{Artin group} associated to $(W,S)$ is defined as
	\[
	A(W) := A(W,S):= \langle S | \underbrace{s_i s_j s_i\dots}_{m_{i,j}\text{ times}} =  \underbrace{s_j s_i s_j\dots}_{m_{i,j}\text{ times}} \text{ for } i,j = 1, \dots n \rangle.
	\]
\end{Definition}

An Artin group $A(W)$ is said to be of \emph{finite type} if its related Coxeter group $W$ is finite.

\begin{Example}

	Consider the Coxeter matrix
	$m=
	\begin{pmatrix}
		1 & 3 & 2 \\
		3 & 1 & 3 \\
		2 & 3 & 1 
	\end{pmatrix}
	$. The equivalent Coxeter graph is depicted in Figure~\ref{fig:Cox_graph}.
	\vspace*{2mm}
	\begin{figure}[h!]
		\centering
		\includegraphics[width=0.35\linewidth]{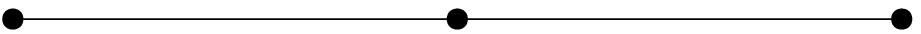}
		\caption{Coxeter graph of $S_4$}
		\label{fig:Cox_graph}
	\end{figure}
	
	The Coxeter group determined by the Coxeter matrix $m$ is the symmetric group $S_4$, which consists of all permutations of the set $\{1,2,3,4\}$. To denote permutations, we use the cycle notation with commas separating the elements in the cycle. A set of Coxeter generators is given by $S = \{s_1, s_2, s_3\}$ with $s_i = (i, i+1)$, the set of all adjacent transpositions. Then the relations are $s_i^2=1$, $s_i s_j s_i = s_j s_i s_j$ if $|i-j|=1$ and $s_i s_j = s_j s_i$ if $|i-j|>1$.
	
	The Artin group associated to $S_4$ is the braid group on $4$ strands, $B_4$. It consists of all braids on $4$ strands and is generated by the braids that swap two adjacent strands, left over right. The generator $b_{2,3}$ is depicted in Figure~\ref{fig:braid}, as an example.
	\vspace*{1mm}
	\begin{figure}[h!]
		\centering
		\includegraphics[width=.118\linewidth]{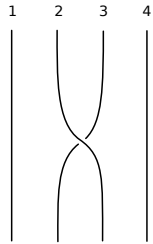}
		\caption{The generator $b_{2,3}$ of $B_4$.}
		\label{fig:braid}
	\end{figure}

\end{Example}

Since the Artin group which is associated to the symmetric group $S_n$ is the braid group on $n$ strands, Artin groups are also called \emph{generalized braid groups}.

Essentially, the Artin group has the same relations as the Coxeter group except that the generators are not self-inverse. Indeed, we even have $s^k \neq I$ for $s \in S$ and all $k \in \N$, where $I$ denotes the identity in the Artin group. To see this, we can construct a homomorphism to $\Z$. For this, consider the map that maps all elements of $S$ to $1 \in \Z$. Then, one can show that this extends to a homomorphism $A \to \Z$. Since the image of $s^k$ is never $0 \in \Z$ for $k \neq 0$, we know that $s^k \neq I$ if $k\neq 0$.

In particular, no Artin group is finite.\\

There exists a natural surjective homomorphism $\phi \colon A(W) \to W$ which maps each generator in $A(W)$ to its counterpart in $W$,
\[
	A(W) \ni s_i \mapsto \phi (s_i) = s_i \in W, \, \forall s_i \in S.
\]
The kernel of this homomorphism is generated by the set $\set{s_i^2 = 1| s_i \in S }$.

In case of the symmetric group and the braid group this homomorphism has a nice depiction. By forgetting how the strands of a braid cross and only looking at the positions of the starting points and end points of the strands, it can be viewed as a permutation. The kernel is then all braids where for each strand starting point and end point are at the same position. Such braids are called \emph{pure braids}.\\

The following definition introduces elements of $W$ which play a special role. Note that we only define them for finite Coxeter groups $W$.

\begin{Definition}

	For $(W,S)$ a Coxeter system of rank $n$ with finite Coxeter group $W$, an element conjugate to $s_1 \dots s_n$ is called a \emph{Coxeter element}.

\end{Definition}

Note that our definition follows Armstrong~\cite{ARM09}, while many authors, as e.g. Humphreys in \cite{HUM90}, define a Coxeter element to be an element of the form $s_{\sigma(1)} \dots s_{\sigma(n)}$, for $\sigma$ a permutation in the symmetric group $S_n$. Since Humphreys showed in Proposition~3.16 of \cite{HUM90} that any two Coxeter elements are conjugate (for his definition) it follows that our definition includes those elements. From this fact also follows that our Coxeter elements form a single conjugacy class.

\newpage

\section{A Lattice in the Coxeter Groups}\label{sec:lat}

\subsection{The Reflection Order}
From now on we only consider finite Coxeter groups $W$.

For a Coxeter system $(W,S)$ we define $T:=\set{wsw^{-1} | w \in W, s \in S}$ to be the conjugacy closure of $S$. We call $T$ the \emph{set of reflections} and an element $\tau \in T$ a \emph{reflection}. An element of $S \subseteq T$ is also called a \emph{simple reflection}.
For $w \in W$ we call $\tau_1 \tau_2 \dots \tau_l = w$ with $\tau_i \in T$ a $T$-decomposition of $w$.

This notion was first introduced by Bessis in \cite{BES03}. He called $(W,T)$ a \emph{dual Coxeter system} and was one of the first to study Coxeter groups with a larger generating set which is closed under conjugation.

\begin{Definition}
	For $w \in W$ let $\ell(w)$ be the minimal number of reflections in a $T$-decomposition of $w$, the \emph{reflection length} of $w$.

\end{Definition}

A $T$-decomposition of an element $w$ using $\ell (w)$ reflections is called a \emph{reduced $T$-decomposition} or – if it is clear from the context – \emph{reduced decomposition}.

The reflection length of $w$ is exactly the geodesic distance of the identity and $w$ on the Cayley graph of $W$ with generating set $T$.

\begin{Lemma}\label{Lem:conj}
	The reflection length is a conjugacy invariant, i.e.
	\[\ell (w) = \ell (u w u^{-1}) \text{ for }  u,w \in W.\]
\end{Lemma}

\begin{proof}
	Let $u,w \in W$ and $w=\tau_1 \dots \tau_k$ be a reduced $T$-decomposition of $w$. Then $k = \ell(w)$ and
	\[ u w u^{-1}= u \tau_1 \dots \tau_k u^{-1} = u \tau_1 u^{-1} \dots u \tau_k u^{-1} .\]
	Now, since $\tau_i \in T$ for all $i \in \{1, \dots, k \}$ and $T$ is closed under conjugation, we have $ u \tau_i u^{-1} \in T$ for all $i$. Thus, $u w u^{-1}$ admits a $T$-decomposition with $k = \ell (w)$ reflections. This shows $\ell(uwu^{-1}) \leq \ell(w)$.
	
	To see $\ell(w) \leq \ell (uwu^{-1})$, note that $w$ is a conjugate of $uwu^{-1}$, so the roles can be swapped.
\end{proof}

\begin{Lemma}\label{Lem:subad}
	The reflection length is sub-additive, i.e.
	\[
		\ell (uw) \leq \ell (u) + \ell (w) \text{ for } u,w \in W .
	\]
\end{Lemma}

\begin{proof}
	Let $u,w \in W$ and $u=\tau_1 \dots \tau_k$, $w=\sigma_1 \dots \sigma_{k'}$ be reduced $T$-decompositions. Then $\ell(u) = k, \ell (w) = k'$ and $\tau_1 \dots \tau_k \sigma_1 \dots \sigma_{k'} = uw$ is a $T$-decomposition of $uw$ and we have $\ell (uw) \leq \ell (u) + \ell (w)$.
\end{proof}

Equality holds whenever $u$ lies on a geodesic from the identity to $uw$ in the Cayley graph of $W$ with generating set $T$, which in turn is the case if and only if there is a shortest $T$-decomposition of $uw$ with a $T$-decomposition of $u$ being a prefix.

Now we define the partial order that will play a central role in the construction of the $K(\pi,1)$s for Artin groups of finite type. Since it is based on the reflection length, we call it \emph{reflection order}.

\begin{Definition}
	Let $\leq$ be the \emph{reflection order} on $W$ defined as
	\[ u \leq w :\iff  \ell(w) = \ell (u) + \ell (u^{-1} w).
	\]
\end{Definition}

In the literature, the reflection order is also called \emph{absolute order}.

It is now easy to observe the following lemma.

\begin{Lemma}
	The group $W$ together with the reflection order $\leq$ is a partially ordered set.
	
\end{Lemma}

\begin{proof}
	
	We will show that the three defining axioms reflexivity, transitivity and anti-symmetry hold. For this, let $u,v,w \in W$.
	\begin{enumerate}[label=(\roman*)]
		\item Since $\ell (w) = \ell(w) + \ell(I) =  \ell (w) + \ell (w^{-1}w)$, we have $ w \leq w$.
	
		\item If $u \leq v, \, v \leq w$, we can deduce
				\[\begin{aligned}[t]
						\ell (w) &= \ell (v) + \ell (v^{-1}w) \\
						&= \ell (u) + \ell (u^{-1}v) + \ell (v^{-1}w)\\
						&\geq \ell (u) + \ell (u^{-1}vv^{-1}w)\\
						&= \ell (u) + \ell (u^{-1}w).
				\end{aligned} \]
				The inequality is due to the sub-additivity of the reflection length (see Lemma~\ref{Lem:subad}). But for the same reason we have $\ell (w) \leq \ell (u) + \ell (u^{-1}w)$, which gives equality and therefore $u \leq w$.
		
		\item If $u \leq w, \, w \leq u$, we have $\ell (w) = \ell (u) + \ell (u^{-1}w) = \ell (w) + \ell (w^{-1}u) + \ell (u^{-1}w)$. But since the reflection length is non-negative we can conclude $\ell (u^{-1}w) = \ell (w^{-1}u) = 0$. Thus, we have $u=w$.
	
	\end{enumerate}

\end{proof}

\subsection{Non-Crossing Partition Lattices}

From now on let $W$ be any finite Coxeter group, $\leq$ the reflection order on it and $\gamma$ an arbitrary Coxeter element in $W$.

Now, as the heading of the section suggests, we want to find a lattice in the poset $(W,\leq)$. Surely, $W$ itself is not necessarily a lattice, since two elements of maximal length have no join. That the reflection length on $W$ is indeed bounded above was first shown by Carter in \cite[Lemma~1-3]{CAR72}. He proved that a maximal element has a reflection length of $|S| = \rk(W)$. Although he only showed this for Weyl groups, the same arguments hold for finite Coxeter groups in general, as noted by Dyer \cite{DYE01}, Bessis \cite{BES03}, Armstrong \cite{ARM09} and many others.

It also follows from Carter's Lemma~3 in \cite{CAR72} that Coxeter elements attain this maximal length.

\begin{Example}

	Consider for example the symmetric group $S_3$ of all permutations of the three-element set $\{1,2,3\}$ with Coxeter generating set $\{(1,2), (2,3)\}$, the set of adjacent transpositions of $S_3$.
	The generating set $T$ is in this case $T_{S_3} = \{(1,2), (1,3), (2,3)\}$ and the Hasse diagram of the reflection order is displayed in Figure~\ref{Fig:Hasse_S3}.

\begin{figure}
	\centering
	\includegraphics[scale=1]{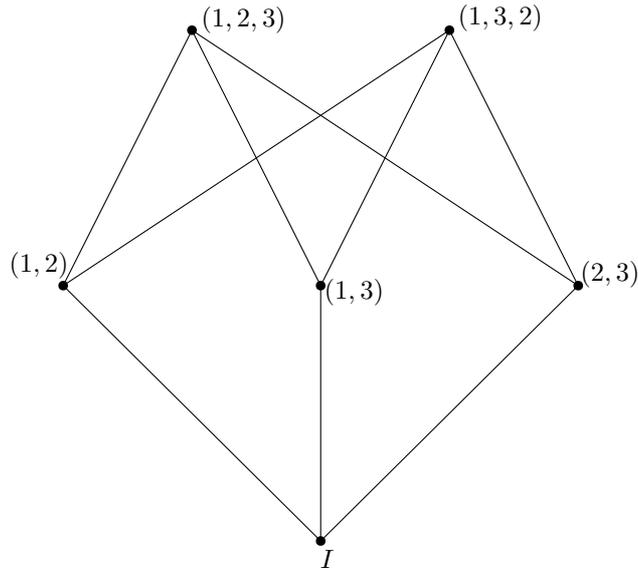}
	\caption{Hasse diagram of the reflection order on $S_3$}
	\label{Fig:Hasse_S3}
\end{figure}

	Obviously the elements $(1,2,3)$ and $(1,3,2)$ are never comparable, since they have the same reflection length of $2$. In particular, they do not have a join.
	
\end{Example}

However, this example might be misleading. Armstrong noted in \cite{ARM09} that although all Coxeter elements are maximal elements of the reflection order, in general not all maximal elements are Coxeter elements. The second implication only holds in the case of type $A$ Coxeter groups, which are the symmetric groups.

\vspace{5mm}

As Brady and Watt have shown in \cite{BW08} and Bessis independently in \cite{BES03}, the subposet $[I,\gamma]$ for $\gamma$ any Coxeter element in $W$ does in fact form a lattice, if $W$ is a finite Coxeter group.

\begin{Theorem}[{\cite[Theorem~7.8]{BW08}}]
\label{thm:lattice}
	
	If $W$ is a finite Coxeter group equipped with the reflection order and $\gamma$ is a Coxeter element, then $[I,\gamma]$ is a lattice.
	
\end{Theorem}

The lattice is also referred to as a non-crossing partition lattice. It is an algebraic generalization of the classical non-crossing partitions, which were first studied by Kreweras in 1972 in \cite{KRE72}. He also proved that they form, ordered by refinement, a lattice.
Biane proved in \cite{BIA97} that in the case of type~$A$ Coxeter groups, the lattice $[I,\gamma]$ coincides with the classical non-crossing partitions, which can be imagined as follows.

Take the set $\{1, \dots , n\}$ and place the elements on a circle, circularly ordered in the natural way. Then the non-crossing partitions of this set are precisely those partitions for which one can draw all partition blocks as convex sets such that no two blocks intersect. For some examples of crossing and non-crossing partitions of the set $\{1, \dots , 7\}$ consider Figure~\ref{fig:NCP}.

\begin{figure}[h]
\centering
   \begin{subfigure}{0.49\linewidth} \centering
     \includegraphics[scale=0.37]{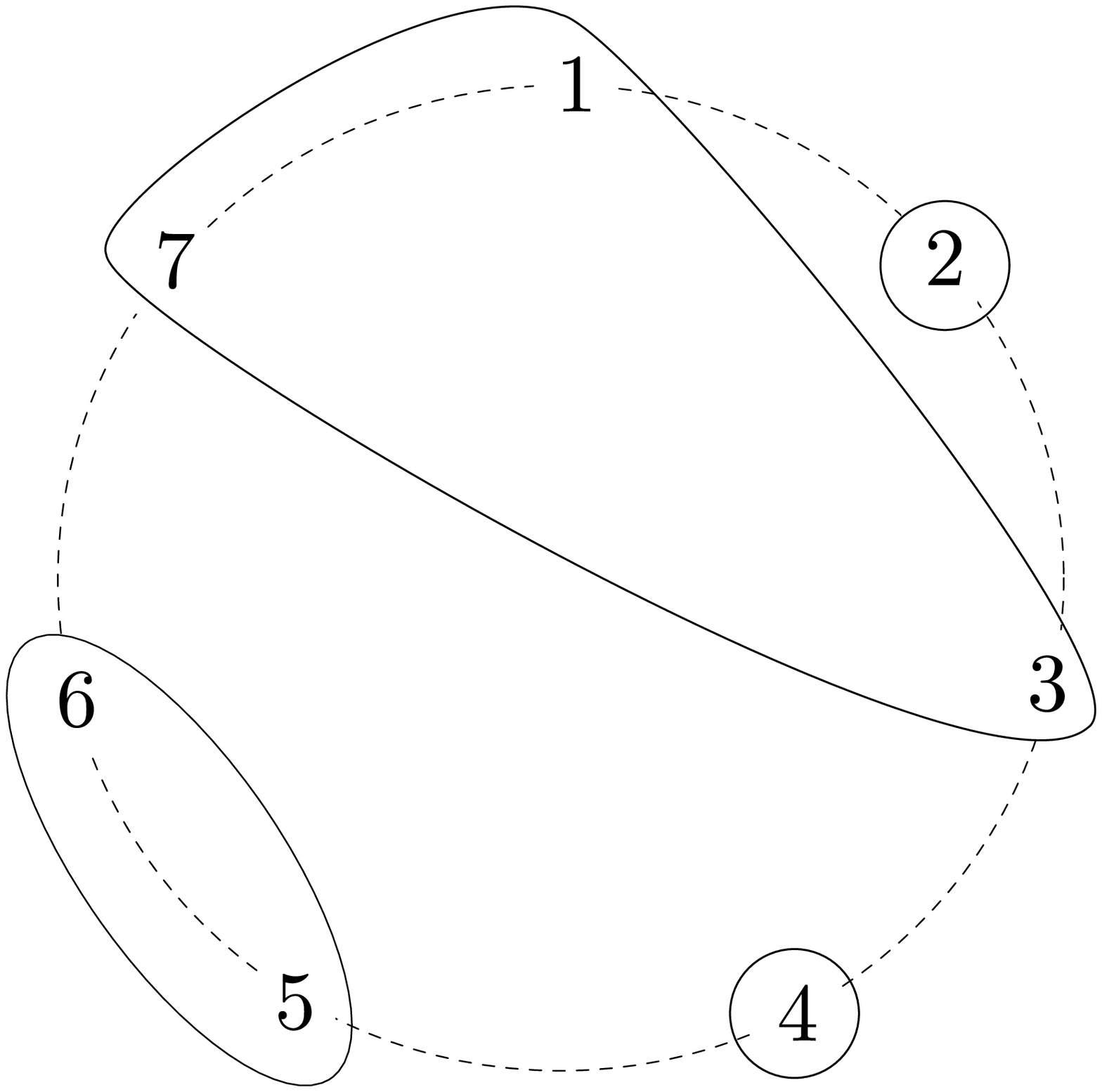}
     \caption{A non-crossing partition}\label{fig:NCP1}
   \end{subfigure}
   \begin{subfigure}{0.49\linewidth} \centering
     \includegraphics[scale=0.37]{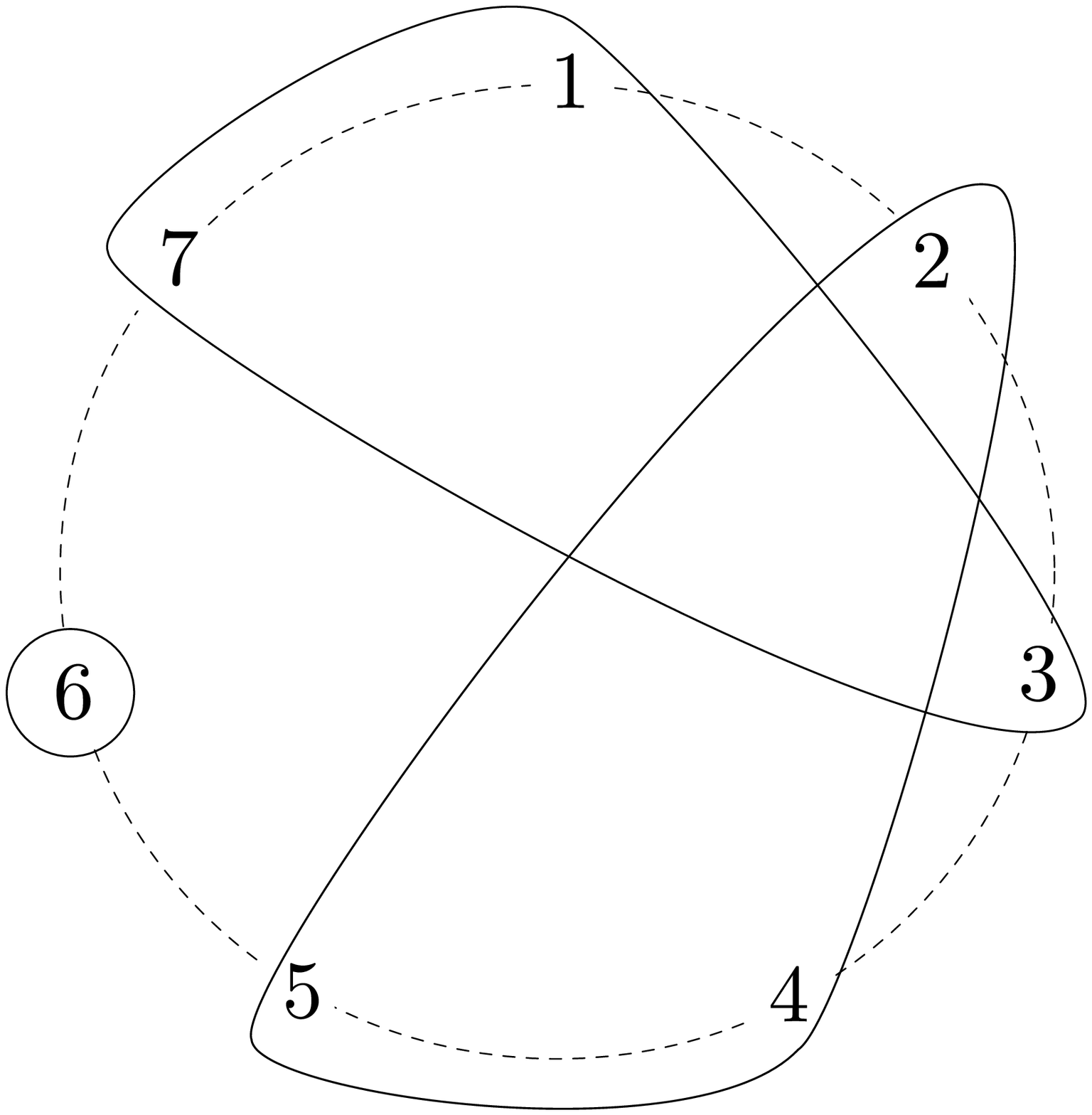}
     \caption{A crossing partition}\label{fig:NCP2}
   \end{subfigure}
\caption{Non-crossing and crossing partition of $\{1, \dots, 7\}$} \label{fig:NCP}
\end{figure}

For this reason we will write $\NC(W, \gamma)$ for the lattice of generalized non-crossing partitions $[I, \gamma]$, for a finite Coxeter group $W$ and a Coxeter element $\gamma$.

\begin{Lemma}
	For two Coxeter elements $\gamma_1, \gamma_2$ in $W$ it holds that
	\[
		\NC(W, \gamma_1) \cong \NC(W, \gamma_2)
	\]
	is a poset-isomorphism.
\end{Lemma}

\begin{proof}
	
	For two Coxeter elements $\gamma_1, \gamma_2 \in W$, there is $v \in W$ such that $\gamma_2 = v \gamma_1 v^{-1}$, since any two Coxeter elements are conjugate.
	Then, for $w \in W$ it holds
	\begin{alignat*}{3}
		& \qquad &w &\in \NC(W, \gamma_1)\\
		&\iff \qquad  &w &\leq \gamma_1\\
	 	&\iff \qquad  &\ell(\gamma_1) &= \ell(w) + \ell(w^{-1}\gamma_1)\\
	 	&\stackrel{\text{\ref{Lem:conj}}}{\iff} \qquad  &\ell(v \gamma_1 v^{-1}) &= \ell(v w v^{-1}) + \ell(v w^{-1} v^{-1} v \gamma_1 v^{-1})\\
	 	&\iff \qquad &\ell(\gamma_2) &= \ell(v w v^{-1}) + \ell((v w v^{-1})^{-1} \gamma_2)\\
	 	&\iff \qquad &vwv^{-1} &\leq \gamma_2\\
	 	&\iff \qquad &vwv^{-1} &\in \NC(W,\gamma_2)
	\end{alignat*}
	
	Thus, conjugation with $v$ maps the elements of $\NC(W, \gamma_1)$ bijectively onto $\NC(W, \gamma_2)$. That this is indeed order-preserving, can be shown in the exact same manner. Since the same holds for the inverse – conjugation with $v^{-1}$ – the lemma is proven.
	
\end{proof}

Because of the previous lemma, we know that the isomorphism type of the non-crossing partition lattice is independent of the choice of a Coxeter element. Therefore we will only refer to it as $\NC(W)$.\\

The following results relate the group structure of $W$ to the poset structure on $W$ given by the reflection order and are taken from \cite{BRA01}, in which Brady constructs the $K(\pi,1)$s for the braid groups.

\begin{Lemma}\label{Lem:3.9}

	Let $u,w \in W$ with $u \leq w$. Then $u^{-1}w \leq w$ and $w u^{-1} \leq w$.

\end{Lemma}

\begin{proof}

	Suppose $u \leq w$ and let $\alpha = u^{-1}w$ and $\beta = w u^{-1}$. Then we have $\ell (w) = \ell (u) + \ell (\alpha)$. Since reflection length is a conjugacy invariant, this implies $\ell (w) = \ell (w^{-1} u w) + \ell (\alpha) = \ell (\alpha^{-1} w) + \ell (\alpha)$. Therefore, $\alpha \leq w$.
	
	For the second inequality note that $\ell (\alpha) = \ell (u \alpha u^{-1}) = \ell (u u^{-1}w u^{-1}) = \ell (\beta)$ and $\ell (u) = \ell (u w^{-1} w) = \ell (\beta^{-1} w)$. This gives $\ell (w) = \ell (u) + \ell (\alpha) = \ell (\beta^{-1} w) +  \ell (\beta)$, which is by definition $\beta \leq w$.

\end{proof}

\begin{Lemma}\label{Lem:3.10}

	Let $u,v,w \in W$ with $u \leq v \leq w$. Then $u^{-1}v \leq u^{-1}w$ and $v^{-1} w \leq u^{-1} w$.

\end{Lemma}

\begin{proof}
	
	Let $\alpha = u^{-1} v$ and $\beta = v^{-1} w$. From $u \leq v$ and $v \leq w$ we get
	\begin{align*}
				v = u \alpha &\text{\quad with \quad } \ell(v) = \ell(u) + \ell(\alpha)\\
				w = v \beta &\text{\quad with \quad } \ell(w) = \ell(v) + \ell(\beta).
	\end{align*}

	From this we get $\ell (w) = \ell(u) + \ell(\alpha) + \ell(\beta)$. Also, we have $u \leq w$ which gives $w = u \alpha \beta$ with $\ell(w) = \ell(u) + \ell(\alpha \beta)$.
	We conclude that $\ell(\alpha \beta) = \ell(\alpha) + \ell(\beta)$, which is by definition $\alpha \leq \alpha \beta$.
	By Lemma \ref{Lem:3.9} we then have $\beta \leq \alpha \beta$. Thus, by definition of $\alpha$ and $\beta$,
	it holds $u^{-1}v = \alpha \leq \alpha \beta = u^{-1} v v^{-1} w = u^{-1}w$ and $v^{-1}w = \beta \leq \alpha \beta = u^{-1} v v^{-1} w = u^{-1} w$. Thus we have established $u^{-1} v \leq u^{-1}w$ and $v^{-1} w \leq u^{-1} w$, as desired.
	
\end{proof}

The following lemma will be used in Lemma~\ref{Lem:cancpreleft} and \ref{Lem:cancpreright} to help establish cancellation properties in a semigroup we are about to define in Section~\ref{Sec:PosetGr}. It shows that the intuition of the structure of a lattice can in fact be transferred to the group structure on the lattice $\NC(W)$.

\begin{Lemma}\label{Lem:3.11}

	Let $u,v,w \in \NC(W)$ and define the elements $a,b,c,d,e,f,g,h,i \in \NC(W)$ by the following equations
	\[
		u \join v = ua = vb, \qquad
		v \join w = vc = wd, \qquad
		u \join w = ue = wf
	\]
	and
	\[
		u \join v \join w = (u \join v ) g = (v \join w) h =  (u \join w) i.
	\]
	Then it holds
	\[
		a \join e = a g = e i, \qquad
		b \join c = b g = c h, \qquad
		d \join f = d h = f i.
	\]
	
\end{Lemma}

\begin{proof}
	The situation is depicted in Figure~\ref{Fig:Lem3.11}. 
	
	\begin{figure}[h!]
		\centering
		\includegraphics[width=0.52\textwidth]{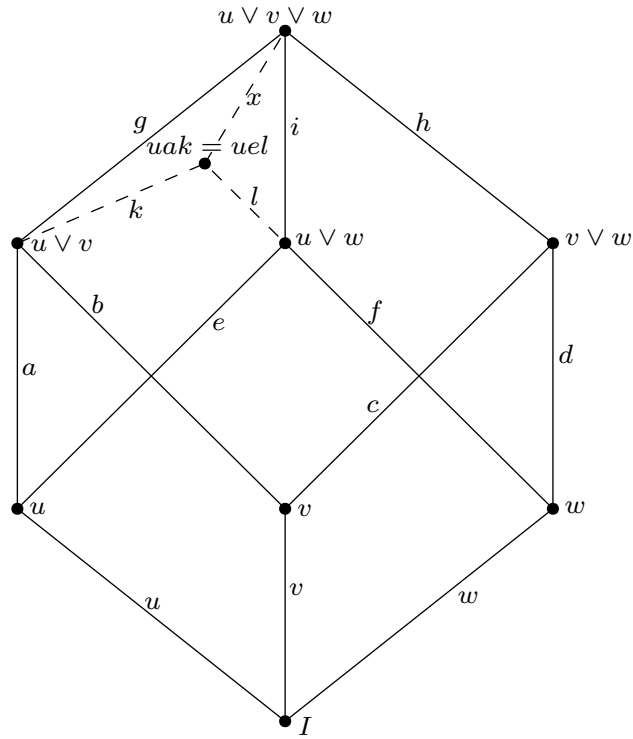}
		\caption{Hasse diagram of the relevant subset}
		\label{Fig:Lem3.11}
	\end{figure}
	
	That the elements $a,b,c,d,e,f,g,h,i \in \NC(W)$ are indeed uniquely determined by the equations above follows from the group structure. To see this, note that if $u a = u a' = u \join v$, then we have $u^{-1} u a = u^{-1} u a'$ and hence $a = a'$. Therefore, there is exactly one element $a = u^{-1} (u \join v)$ with $ua = u \join v$.\\
	
	We only prove the first equation, the other two are similar.
	We know that $u \join v \join w = (u \join v) g = u a g $ and  $u \join v \join w = (u \join w) i =  u e i$.
	Thus, $u a g = u e i$, which is equivalent to $a g = e i$.
	
	What is left to show is that this is exactly the join of $a$ and $e$.
	For this, note that by definition of the join we have $u \leq u \join v = u a \leq (u \join v) \join w = u a g$.
	Then we can apply Lemma \ref{Lem:3.10} to deduce that $u^{-1} u a \leq u^{-1} u a g$, which is $a \leq a g = e i$.
	In the same way we get $e \leq e i = a g$. From this follows that $a \join e \leq ag =ei$.
	
	Now assume, for the sake of contradiction, that $a \join e < ag = ei$, say $a \join e = ak = el$ and define the elements $x,y$ by the equation
	$akx=ag = ei = elx$. This is again depicted in Figure~\ref{Fig:Lem3.11}. Then it holds $a \leq ak \leq akx$, from which with Lemma~\ref{Lem:3.10} follows that $k \leq kx$. This is equivalent to $\ell(kx) = \ell(k) + \ell(x)$. We want to show that then $uak = uel$ is a smaller upper bound for $u , v,w $ than $uag=uei$, which contradicts the minimality of the join of $u,v,w$.
	
	We know that $u \leq ua \leq uag = uakx$ and therefore $\ell(uakx) = \ell(ua) + \ell(kx)$. From the sub-additivity of the reflection length it follows that
	\begin{align*}
		\centering
		\ell(uakx) &\leq \ell(uak) + \ell(x)\\
		\iff \ell(uak) &\geq \ell(uakx) - \ell(x)\\
		\iff \ell(uak) &\geq \ell(ua) + \ell(kx) - \ell(x)\\
		\iff \ell(uak) &\geq \ell(ua) + \ell(k) + \ell(x) - \ell(x)\\
		\iff \ell(uak) &\geq \ell(ua) + \ell(k)
	\end{align*}
		
	Since $\ell(uak) \leq \ell(ua) + \ell(k)$ is true by the sub-additivity of $\ell$, we have established $ua \leq uak$. Therefore $uak$ is an upper bound for $u a = u \join v$ and with that, one for $u$ and $v$.
	Applying this argument onto $uelx$ instead of $uakx$ gives $u,w \leq uel = uak$. This contradicts the minimality of $u \join v \join w = uag =uei$ and thus we have shown $a\join e = ag = ei$.

\end{proof}

\newpage

\section{Poset Groups}\label{Sec:PosetGr}

From now on let $W$ be any finite Coxeter group and $\gamma$ a fixed but arbitrary Coxeter element. Knowing that $\NC(W)$ forms a lattice, we will construct the so called poset group. We will follow the construction of Brady from \cite{BRA01} which is also used in \cite{BW02} from Brady and Watt.

\begin{Definition}
	
	Let the \emph{poset group} $\Gamma = \Gamma(W,\gamma)$ be the following group.\\
	\noindent
	For each element $w \in \NC(W) \backslash I$ we take one formal generator $\bracket{w}$. The group relations are of the form $\bracket{w_1} \bracket{w_2} = \bracket{w_3}$ whenever $w_1 w_2 = w_3$ in $W$ and $\ell(w_1) + \ell(w_2) = \ell(w_3)$.
	
\end{Definition}

In other words, for every relation $w_1 \leq w_3$ in the lattice $\NC(W)$ we have a relation $\bracket{w_1} \bracket{w_1^{-1} w_3} = \bracket{w_3}$ in $\Gamma$.

We do not explicitly define what $\bracket{I}$ is, but it will sometimes occur when we consider words in the generators which represent elements of $\Gamma$. In this case we regard it as the empty word.\\

The group constructed here is exactly the group which Bessis called $\mathbf{G}(P_c)$ in \cite{BES03}.\\

Starting with this poset group $\Gamma = \Gamma (W,\gamma)$ we construct a $K(\Gamma,1)$. To obtain a $K(A(W),1)$ for the Artin group, we will then show that the related Artin group $A(W)$ is isomorphic to the poset group $\Gamma$. Therefore, the following theorem, which is taken from \cite{BES03}, is crucial for our construction. Bessis uses a case-by-case proof, which is partially achieved by computer. Up to now, no case free proof is known for this fact.

\begin{Theorem}[{\cite[Theorem~2.2.5]{BES03}}]\label{thm:iso}
	Let $W$ be a finite Coxeter group and $\gamma \in W$ a Coxeter element. Then
	\[
		\Gamma(W,\gamma) \cong A(W).
	\]
\end{Theorem}
\vspace{6mm}

\subsection{A Cancellative Semi Group in the Poset Group}

Since the relations that are used in the definition of $\Gamma$ do not involve inverses of generators, we can use the very same presentation to define a semigroup $\Gamma_+ = \Gamma_+(W,\gamma)$. The following definitions and notions were introduced by Garside in \cite{GAR69}, who used them to give a solution to the conjugacy problem in the braid groups.\\

A \emph{positive word} is a word in the generators of $\Gamma$ that does not involve an inverse. Positive words represent elements of $\Gamma_+$.
Two positive words $A, B$ are \emph{positively equal} if there exists a sequence of positive words $A_0, \dots, A_k$, such that $A_0 = A, A_k = B$ and $A_i$ is obtained from $A_{i-1}$ by replacing one side of a defining relation by the other. If $A$ and $B$ are positively equal we also write $A \doteq B$.

If two words $A$ and $B$ in the generators of $\Gamma$ are identical letter by letter, we write $A \equiv B$. If two positive words are identical, they are also positively equal.\\

The reflection length on elements of $\NC(W)$ can be used to associate a \emph{length} to each generator of $\Gamma$. And since the semigroup $\Gamma_+$ is only defined by relations of the form $\bracket{ w_1} \bracket{w_2} = \bracket{w_3}$ whenever $w_1 w_2 = w_3$ in $W$ and $\ell (w_1) + \ell (w_2) = \ell (w_3)$, the relations in $\Gamma_+$ also preserve the length and we can associate to each positive word $A$ a length $\ell(A)$. For the same reason two positively equal words must have the same length.\\

Our next goal is to establish cancellation properties in $\Gamma_+$ and then show that $\Gamma_+$ embeds in $\Gamma$. Therefor we first need some lemmas.
\vspace*{-1mm}
\begin{Lemma}\label{Lem:distr}
	
	Let $v, b, c \in \NC(W)$ with $v \leq vb, v \leq vc$ and $b \join c = bg = ch$ for $g,h \in \NC(W)$. Then $vb \join vc = vbg = vch$.

\end{Lemma}

\begin{proof}

	Assume to the contrary that  $vb \join vc = vbk = vcm < vbg = vch$. Let $x$ be such that $vbkx = vbg$. The situation is depicted in Figure~\ref{Fig:Lem4.3}. Then we have $v \leq vb \leq vbk$. By Lemma~\ref{Lem:3.10} we then get $b \leq bk$. Applying Lemma~\ref{Lem:3.10} to $v \leq vbk \leq vbkx = vbg$ gives $bk \leq bkx=bg$ and since $x$ is not the identity we have $bk<bg = b \join c$.
	
	Using the same arguments on $vcm$ we can show that $c \leq cm$ and $cm < ch = b \join c$. But then we have found an element $bk=cm$ which is greater than $b$ and $c$ but strictly less than the join of $b$ and $c$. This contradicts the minimality of the join of $b$ and $c$.
	
	\begin{figure}[h]
		\centering
		\includegraphics[width=0.4\textwidth]{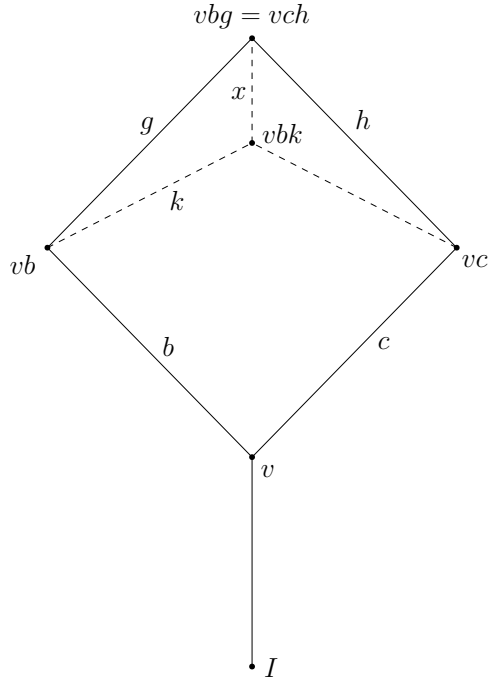}
		\caption{Hasse diagram of the subposet from Lemma~\ref{Lem:distr}}
		\label{Fig:Lem4.3}
	\end{figure}
	
\end{proof}
\vspace*{-4mm}
\begin{Lemma}\label{Lem:cancpreleft}

	Let $u,w \in \NC(W)$ and $A,B$ be positive words in $\Gamma$.
	
	\noindent
	If $\bracket{u} A \doteq \bracket{w} B$ then there exist $e,f \in \NC(W)$ and a positive word $C$, such that \[A \doteq \bracket{e} C, \qquad B \doteq \bracket{f} C\]
	\vspace*{-4mm}
	and $u \join w = ue = wf$.

\end{Lemma}

\begin{proof}
	
	Following the proof of Brady in \cite{BRA01}, who himself refers to Garside \cite{GAR69}, we do double induction on firstly, the length of $\bracket{u} A$ – which is the same as the length of $\bracket{w} B$ – and secondly on the number of substitutions in the sequence of positive words which realizes $\bracket{u} A \doteq \bracket{w} B$. We denote the length of $\bracket{u} A$ as $n$ and the number of substitutions as $m$.\\
	
	For the induction base we will prove that the lemma holds for length $n=1$ and an arbitrary number of substitutions $m \in \N$. And we will also show that the lemma holds for every word length $n \in \N$ if there is only one substitution in the sequence of positive words which realizes $\bracket{u} A \doteq \bracket{w} B$.
	After that we are prepared to do the induction step.
	
	For the induction step we fix arbitrary $n_0,m_0 \in \N$ with $n_0,m_0 > 1$. We then assume that the lemma holds for length $n_0$ if the number of substitutions in the sequence is $m < m_0$. And we assume that if the length of $\bracket{u} A$ is $n < n_0$ that it holds for every number of substitutions $m$. We then show that it also holds for length $n_0$ and $m_0$ substitutions.
	
	An Illustration of the structure of the double induction is depicted in Figure~\ref{fig:sketch}.\\
	
	\begin{figure}[h!]
		\centering
		\begin{subfigure}{0.3\linewidth}
			\centering
     		\includegraphics[scale=0.7]{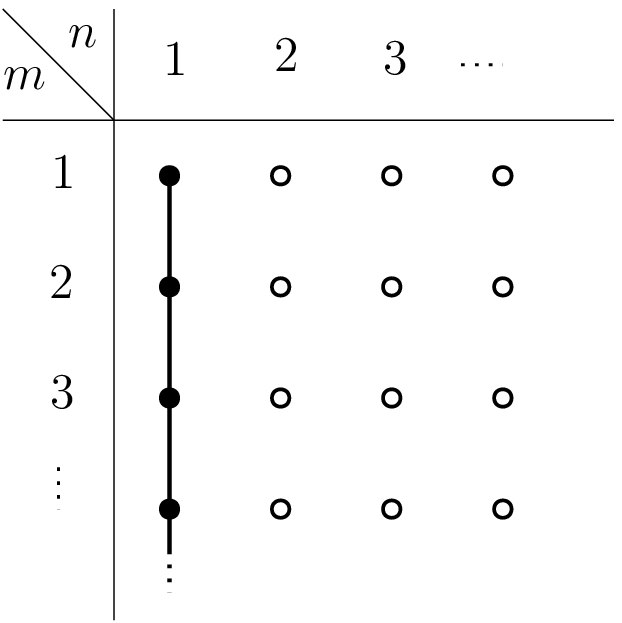}
     		\caption{First part of the induction base}\label{fig:sketch1}
     	\end{subfigure}
     	\hspace*{2mm}
     	\begin{subfigure}{0.3\linewidth}
     		\centering
     		\includegraphics[scale=0.7]{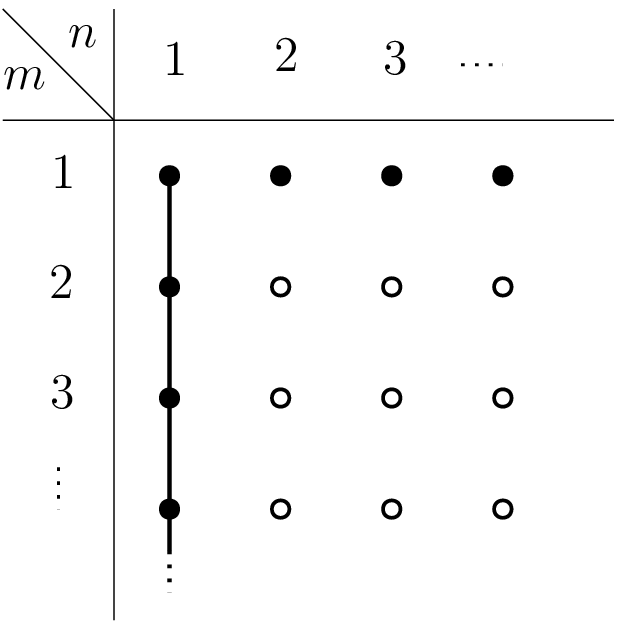}
     		\caption{Second part of the induction base}\label{fig:sketch2}
     	\end{subfigure}
     	\hspace*{2mm}
     	\begin{subfigure}{0.3\linewidth}
			\centering
     		\includegraphics[scale=0.7]{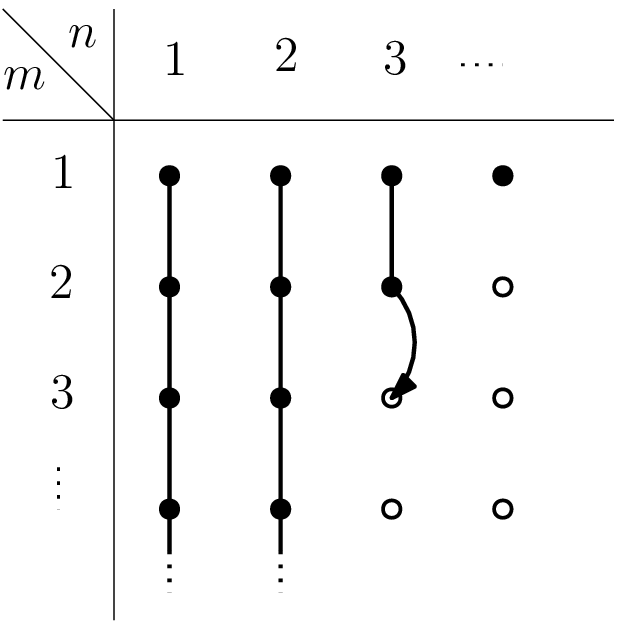}
     		\caption{Induction step for $n_0=m_0=3$}\label{fig:sketch3}
     	\end{subfigure}
     	\caption{Illustration of the structure of the proof} \label{fig:sketch}
	\end{figure}
	
	For the first part of the induction base let $n = 1$ and let the realizing sequence contain an arbitrary number of substitutions. Since $\ell (\bracket{u} A) = 1 = \ell (\bracket{w} B)$, both $\bracket{u} A$ and $\bracket{w} B$ consist of only one formal generator of $\Gamma$, which must come from a reflection in $\NC(W)$. Thus, we have $\bracket{u} A = \bracket{t_1}, \bracket{w} B = \bracket{t_2}$ for $t_1,t_2$ reflections in $\NC(W)$.
	
	Now, since $\bracket{t_1} \doteq \bracket{t_2}$, we already have $t_1 = t_2$, because any relation that could be applied on the left side is of the form $\bracket{t_1} = \bracket{w} \bracket{w^{-1} t_1}$ for $w \leq t_1$. But this is only true for $I \leq t_1$ and $t_1 \leq t_1$. Thus, a choice of $e,f \in \NC(W)$ and a positive word $C$ is possible such that the theorem is true for word length $1$.\\
	
	Now we want to show that the theorem holds for arbitrary length of $\bracket{u} A, \bracket{w} B$ when there is only one substitution in the sequence of positive words which realizes $\bracket{u} A \doteq \bracket{w} B$. In this case, we have
	\[
		\bracket{u} A \equiv \bracket{a_0} \bracket{a_1} \dots \bracket{a_k}
	\]
	for $a_0, \dots ,a_k \in \NC(W)$ with $a_0 = u$ and $k \in \N_0$, as well as
	\[
		\bracket{w} B \equiv \bracket{b_0} \bracket{b_1} \dots \bracket{b_{k'}}
	\]
	for $b_0, \dots, b_{k'} \in \NC(W)$ with $b_0 = w$ and $k' \in \N_0$.
	Without loss of generality let $k'=k-1$ and let the single substitution be $\bracket{a_i} \bracket{a_{i+1}} = \bracket{b_i}$ for some $i \in \{0, \dots k-1\}$.
	Then we know that
	\begin{equation}\label{eq:equiv}
		\bracket{a_0} \dots \bracket{a_{i-1}} \bracket{b_i} \bracket{a_{i+2}} \dots \bracket{a_k} \equiv \bracket{b_0} \bracket{b_1} \dots \bracket{b_{k-1}}.
	\end{equation}
	We distinguish two cases.
	
	The first case $i=0$ means that $\bracket{u} \bracket{a_1} = \bracket{w}$. Since this is a relation, it must hold $u \leq w$ with $a_1 = u^{-1}w$ and hence $u \join w = w$. Setting $e = a_1 = u^{-1}w, f = I$ and $C \equiv \bracket{b_1} \dots \bracket{b_{k'}} \equiv \bracket{a_2} \dots \bracket{a_k}$, we have $ue = wf = w = u \join w$ and $A \equiv \bracket{e} C, B \equiv \bracket{f}C$.
	
	In the second case we have $i \geq 1$. Then, it follows from Equation \ref{eq:equiv} that $u=w$. Setting $e=f=I$, we get $ue = wf =u =w = u \join w$.
	But we also know that $A \equiv \bracket{a_1} \dots \bracket{a_k} \doteq \bracket{a_1} \dots \bracket{a_{i-1}} \bracket{b_i} \bracket{a_{i+2}} \dots \bracket{a_k} \equiv \bracket{b_1} \dots \bracket{b_{k-1}} \equiv B$, since $\bracket{a_i} \bracket{a_{i+1}} = \bracket{b_i}$ for some $i \geq 1$. Thus, setting $C \equiv B$ gives $A \doteq \bracket{e} C, B \equiv \bracket{f}C$ and completes the induction base.\\
	
	Now, for the induction step let $n_0$ be the length of $\bracket{u} A$ and $m_0$ be the number of substitutions in the sequence between $\bracket{u} A$ and $\bracket{w} B$ with $n_0, m_0 > 1 $ arbitrary.
	
	Choose some expression $\bracket{v} D$ in the sequence between $\bracket{u} A$ and $\bracket{w} B$ where $v$ is not the identity.
	Then
	\[\bracket{u} A \doteq \bracket{v} D \text{ and } \bracket{v} D \doteq \bracket{w} B\]
	and both realizing sequences use less substitutions than the sequence which realizes $\bracket{u} A \doteq \bracket{w} B$. Therefore we have by induction
	\[A \doteq \bracket{a} E, \, D \doteq \bracket{b} E \text{ and } B \doteq \bracket{d} F, \, D \doteq \bracket{c} F\]
	for positive words $E, F$ and  $a,b,c,d \in \NC(W)$ with $u \join v = ua = vb$ and $v \join w = vc = wd$.
	We also get $\bracket{b} E \doteq \bracket{c} F$ with length strictly less than the length of $\bracket{u} A$ since $v$ is not the identity. Then, we can use induction on the word length $n$ to get
	\[
	E \doteq \bracket{g} G, \, F \doteq \bracket{h} G
	\]
	for a positive word $G$ and $g,h \in \NC(W)$ with $b \join c = bg = ch$. Then, we have $A \doteq \bracket{a} \bracket{g} G$ and $B \doteq \bracket{d} \bracket{h} G$.
	We also get $u \join v \join w = (u \join v) \join (v \join w) = vb \join vc = vbg = vch$ by Lemma~\ref{Lem:distr}. If we define $e,f,i \in \NC(W)$ by
	$u \join w = ue = wf$ and $u \join v \join w = (u \join w) i$, we are in the situation of Lemma~\ref{Lem:3.11}. Thus we can conclude $a \join e = ag = ei$ and $d \join f = dh = fi$.
	
	So we have $A \doteq \bracket{e} \bracket{i} G$ and $B \doteq \bracket{f} \bracket{i} G$. Setting $C = \bracket{i} G$ finishes the proof.
	
\end{proof}

If we replace the prefixes in the proof of Lemma~\ref{Lem:cancpreleft} by suffixes, similar arguments can be used to get the following result.

\begin{Lemma}\label{Lem:cancpreright}

	Let $u,w \in \NC(W)$ and $A,B$ be positive words in $\Gamma$.
	
	\noindent
	If $A \bracket{u} \doteq B \bracket{w}$ then there exist $e,f \in \NC(W)$ and a positive word $C$, such that \[A \doteq C \bracket{e}, \qquad B \doteq C \bracket{f}\]
	and $u \join w = eu = fw$.

\end{Lemma}

\begin{Corollary}\label{Cor:canc}

	The semigroup $\Gamma_+$ has left and right cancellation properties, i.e.
	
	for $A,B \in \Gamma_+$ and $u,w \in \NC(W)$ we have the following.
	
	\begin{enumerate}[label=\normalfont{(\arabic*)}]
	
		\item If $\bracket{u} A \doteq \bracket{u} B$, then $A \doteq B$ and
		
		\item if $A \bracket{u} \doteq B \bracket{u}$, then $A \doteq B$.
	
	\end{enumerate}

\end{Corollary}

\begin{proof}~

	\begin{enumerate}[label=(\arabic*)]
		\item \label{case1} If $\bracket{u} A \doteq \bracket{u} B$, we are in the situation of Lemma~\ref{Lem:cancpreleft}. Since $u \join u = u	$, the two elements from Lemma \ref{Lem:cancpreleft}, $e$ and $f$, must be the identity and we have $A \doteq C$ and $B \doteq C$ for a positive word $C$. Thus, we have $A \doteq B$.
		
		\item This case uses the same argument as \ref{case1} but in the situation of Lemma~\ref{Lem:cancpreright}. 
	
	\end{enumerate}
	
\end{proof}

\begin{Corollary}\label{Cor:cancleft}

	Let $a_1, a_2, \dots, a_k \in \NC(W)$ and $P, X_1, X_2, \dots X_k$ be positive words in~$\Gamma$.
	
	If $P \doteq \bracket{a_1} X_1 \doteq \dots \doteq \bracket{a_k} X_k$, then there exists a positive word $Z$ such that
	\[
		P \doteq \bracket{a_1 \join \dots \join a_k} Z.
	\]
	
\end{Corollary}

\begin{proof}
	The proof goes by induction on $k$.
	
	In the case $k=2$ we have $P \doteq \bracket{a_1} X_1 \doteq \bracket{a_2} X_2$. Then we can use Lemma~\ref{Lem:cancpreleft} and get
	$b_1,b_2 \in \NC(W)$ with $a_1 \join a_2 = a_1 b_1 = a_2 b_2$ and a positive word $Z$ with $X_1 \doteq \bracket{b_1} Z$ and $X_2 \doteq \bracket{b_2} Z$. Thus, we have $P \doteq \bracket{a_1} \bracket{b_1} Z \doteq \bracket{a_2} \bracket{b_2} Z$. But since $a_1 \leq a_1 b_1$ this is equivalent to $P \doteq \bracket{a_1 b_1} Z$. And the right side of the last equation is just $\bracket{a_1 \join a_2} Z$.
	
	For the induction step let $k > 2$. Considering only the first $k-1$ equalities of $P \doteq \bracket{a_1} X_1 \doteq \dots \doteq \bracket{a_k} X_k$, we get, by induction, a positive word $Y$ with $P \doteq \bracket{a_1 \join \dots \join a_{k-1}} Y$. Using the same argument as in the induction base on the equation $P \doteq \bracket{a_1 \join \dots \join a_{k-1}} Y \doteq \bracket{a_k} X_k$ we get $P \doteq \bracket{ (a_1 \join \dots \join a_{k-1} ) \join a_k} Z$ for a positive word $Z$.

\end{proof}

Again, using Lemma~\ref{Lem:cancpreright} in the previous proof we can achieve the same result for suffixes.

\begin{Corollary}\label{Cor:cancright}
	
	Let $a_1, a_2, \dots, a_k \in \NC(W)$ and $P, X_1, X_2, \dots X_k$ be positive words in~$\Gamma$.
	
	If $P \doteq X_1 \bracket{a_1} \doteq \dots \doteq X_k \bracket{a_k}$, then there exists a positive word $Z$ such that
	\[
		P \doteq Z \bracket{a_1 \join \dots \join a_k}.
	\]
	
\end{Corollary}


\subsection{The Semi Group Embeds}

Our next step to constructing $K(\pi,1)$s is to show that $\Gamma_+$ embeds into $\Gamma$. In other words, we want to show that if two positive words in $\Gamma$ are equal, then they are positively equal. This means that if there exists some sequence of words in $\Gamma$ transforming a positive word into another positive word by using the group relations, then we can also find such a sequence which only consists of positive words.

For this we need the following lemma from Brady \cite{BRA01}.

\begin{Lemma}\label{Lem:pos}
	Let $w \in \NC(W)$. Then $\gamma w^{-1} \in \NC(W)$ and $\gamma w \gamma^{-1} \in \NC(W)$ and the following identities hold
	
	\begin{enumerate}[label=\normalfont{(\alph*)}]

		\item \label{eq:Lemmaa} $\bracket{\gamma} \doteq \bracket{\gamma w^{-1}} \bracket{w}$
		
		\item \label{eq:Lemmab} $\bracket{\gamma} \doteq \bracket{\gamma w \gamma^{-1}} \bracket{\gamma w^{-1}}$
		
		\item \label{eq:Lemmac} $\bracket{\gamma} \bracket{w} \doteq \bracket{\gamma w \gamma^{-1}} \bracket{\gamma}$.
		
	\end{enumerate}
	
\end{Lemma}

\begin{proof}
	Since $w \leq \gamma$ we have $\gamma w^{-1} \leq \gamma$ by Lemma~\ref{Lem:3.9} and thus by definition of $\NC(W)$ $\gamma w^{-1} \in \NC(W)$. Applying the same lemma to $\gamma w^{-1} \leq \gamma$ gives $\gamma (\gamma w^{-1})^{-1} = \gamma w \gamma^{-1} \in \NC(W)$.
	
	Since $\gamma w^{-1} \leq \gamma$ we have $\ell (\gamma) = \ell (\gamma w^{-1}) + \ell ((\gamma w^{-1})^{-1} \gamma) = l(\gamma w^{-1}) + l(w)$. Therefore, by definition of $\Gamma$, Equation~\ref{eq:Lemmaa} is a relation in $\Gamma$.
	
	The same arguments for $\gamma w \gamma^{-1} \leq \gamma$ provide
	$\ell (\gamma) = \ell (\gamma w \gamma^{-1}) + \ell ((\gamma w \gamma^{-1})^{-1} \gamma)
	= \ell (\gamma w \gamma^{-1}) + \ell (\gamma w^{-1}) $ and thus by definition of $\Gamma$, Equation~\ref{eq:Lemmab} must also be a relation in $\Gamma$.
	To obtain Equation~\ref{eq:Lemmac}, simply use \ref{eq:Lemmab} on the left factor of $\bracket{\gamma} \bracket{w}$ to get $\bracket{\gamma w \gamma^{-1}} \bracket{\gamma w^{-1}} \bracket{w}$ and then apply \ref{eq:Lemmaa} to the second and third factor to get
	$\bracket{\gamma w \gamma^{-1}} \bracket{\gamma}$. All the transformations were positively equal and therefore we have established Equation~\ref{eq:Lemmac}.
\end{proof}

A direct consequence of this lemma is the following corollary. It results from applying Lemma~\ref{Lem:pos} inductively.

\begin{Corollary} \label{Cor:pos}
	Let $w \in \NC(W)$. Then the following identities hold for $k \in \N_0$.
	
	\begin{enumerate}[label=\normalfont{(\alph*)}]
		\item \label{eq:Cora} $\bracket{\gamma} \doteq \bracket{\gamma^{k+1} w^{-1} \gamma^{-k}} \bracket{\gamma^k w \gamma^{-k}}$
		
		\item \label{eq:Corb} $\bracket{\gamma} \doteq \bracket{\gamma^{k+1} w \gamma^{-(k+1)}} \bracket{\gamma^{k+1} w^{-1} \gamma^{-k}}$
		
		\item \label{eq:Corc} $\bracket{\gamma} \bracket{\gamma^k w \gamma^{-k}} \doteq \bracket{\gamma^{k+1} w \gamma^{-(k+1)}} \bracket{\gamma}$.
		
	\end{enumerate}
	
\end{Corollary}

\begin{proof}
	First we will show by induction on $k$ that $\gamma^k w \gamma^{-k} \in \NC(W)$ for $k \in \N$, if $w \in \NC(W)$.
	
	Lemma~\ref{Lem:pos} gives the induction base. For the induction step assume that we have $u := \gamma^{k-1} w \gamma^{-(k-1)} \in \NC(W)$ for some $k>1$. Then, again, by Lemma~\ref{Lem:pos} we have $\gamma u \gamma^{-1} = \gamma^k w \gamma^{-k} \in \NC(W)$.
	
	We can now apply \ref{eq:Lemmaa}, \ref{eq:Lemmab} and \ref{eq:Lemmac} of Lemma~\ref{Lem:pos} to $v:=\gamma^k w \gamma^{-k}$ for $k \in \N_0$ and some $w \in \NC(W)$ to obtain
	\begin{enumerate}[label=\normalfont{(\alph*)}]
		
		\item $\bracket{\gamma} \doteq \bracket{\gamma v^{-1}} \bracket{v} \doteq \bracket{\gamma^{k+1} w^{-1} \gamma^{-k}} \bracket{\gamma^k w \gamma^{-k}}$
		
		\item $\bracket{\gamma} \doteq \bracket{\gamma v \gamma^{-1}} \bracket{\gamma v^{-1}} \doteq \bracket{\gamma^{k+1} w \gamma^{-(k+1)}} \bracket{\gamma^{k+1} w^{-1} \gamma^{-k}}$
		
		\item $\bracket{\gamma} \bracket{v} \doteq \bracket{\gamma v \gamma^{-1}} \bracket{\gamma} \doteq \bracket{\gamma^{k+1} w \gamma^{-(k+1)}} \bracket{\gamma}$.
		
	\end{enumerate}
\end{proof}

Now we are ready to proof the following theorem.

\begin{Theorem}[{\cite[Theorem~5.7]{BRA01}}]\label{thm:embed}
	
	The semigroup $\Gamma_+$ embeds in $\Gamma$, i.e. if two positive words are equal, then they are positively equal.
	
\end{Theorem}

\begin{proof}
	
	Take two positive words, say $X$ and $Y$, that are equal. Then there exists a sequence of words
	\begin{equation} \label{seq1}
	X=W_0, W_1, \dots, W_k=Y
	\end{equation}
	s.t. $W_{i+1}$ is obtained from $W_i$ either by applying one of the defining relations $\bracket{w_1} = \bracket{w_2} \bracket{w_3}$ or by introducing or deleting either $\bracket{w} \bracket{w}^{-1}$ or $\bracket{w}^{-1} \bracket{w}$. Let $m_i$ be the number of occurrences of inverses of generators in $W_i$ ($i=0,\dots ,k$) and let $m$ be the maximum of all $m_i$.
Instead of explicitly giving a sequence of positive words from $X$ to $Y$, we will construct such one from $\bracket{\gamma}^m X$ to $\bracket{\gamma}^m Y$. This will show $\bracket{\gamma}^m X \doteq \bracket{\gamma}^m Y$ and we can use left cancellation (Corollary~\ref{Cor:canc}) to conclude $X \doteq Y$.
\ \\

	We start with $\bracket{\gamma}^m X= \bracket{\gamma}^m W_0$ and inductively append $k$ sequences of positive words. The sequence appended in step $i$ depends on the relation that was used to obtain $W_i$ from $W_{i-1}$ in sequence \ref{seq1}. The last word of each sequence will take the role of $W_i$ and be of the form $\bracket{\gamma}^{m-m_i} \bracket{w_1} \bracket{w_2} \dots$ for some elements $w_1, w_2, \dots \in \NC(W)$.

	So assume that to obtain $W_i$ from $W_{i-1}$ there was $\bracket{w} \bracket{w}^{-1}$ introduced after the $j$-th formal generator in the word $W_{i-1}$. Let $\bracket{\gamma}^{m-m_{i-1}} \bracket{w_1} \bracket{w_2} \dots$ be the last word of the previously added sequence. Let $k \in \N_0$, s.t. the inverse added in step $i$ is the $(k+1)$-th inverse from the right in $W_i$. Then we append the following sequence.
	
	\begin{align*}		
		\bracket{\gamma}^{m-m_{i-1}-1} \bracket{\gamma w_1 \gamma^{-1}} \bracket{\gamma} \bracket{w_2} \dots \bracket{w_j} &\bracket{w_{j+1}} \dots \\
		\bracket{\gamma}^{m-m_{i-1}-1} \bracket{\gamma w_1 \gamma^{-1}} \bracket{\gamma w_2 \gamma^{-1}} \bracket{\gamma} \dots \bracket{w_j }&\bracket{w_{j+1}} \dots \\
		\vdots \hspace{23mm}\\
		\bracket{\gamma}^{m-m_{i-1}-1} \bracket{\gamma w_1 \gamma^{-1}} \bracket{\gamma w_2 \gamma^{-1}} \dots \bracket{ \gamma w_j \gamma^{-1}} \bracket{&\gamma} \bracket{w_{j+1}} \dots\\
		\bracket{\gamma}^{m-(m_{i-1}+1)} \bracket{\gamma w_1 \gamma^{-1}} \dots \bracket{ \gamma w_j \gamma^{-1}} \bracket{\gamma^{k+1} w \gamma^{-(k+1)}} &\bracket{\gamma^{k+1} w^{-1} \gamma^{-k}} \bracket{w_{j+1}} \dots~,
	\end{align*}
	
	i.e. we take one copy of $\bracket{\gamma}$ and use \ref{eq:Lemmac} from Lemma \ref{Lem:pos} to bring it to the $j$-th position and then use \ref{eq:Corb} from Corollary \ref{Cor:pos} to absorb the inverse. Since $m_i = m_{i-1} +1 $ the last word has the required form.
	
	If instead $W_i$ was obtained from $W_{i-1}$ by introducing $\bracket{w}^{-1} \bracket{w}$, we do the same with the sequence
	
	\begin{align*}		
		\bracket{\gamma}^{m-m_{i-1}-1} \bracket{\gamma w_1 \gamma^{-1}} \bracket{\gamma} \bracket{w_2} \dots \bracket{w_j} &\bracket{w_{j+1}} \dots \\
		\bracket{\gamma}^{m-m_{i-1}-1} \bracket{\gamma w_1 \gamma^{-1}} \bracket{\gamma w_2 \gamma^{-1}} \bracket{\gamma} \dots \bracket{w_j }&\bracket{w_{j+1}} \dots \\
		\vdots \hspace{23mm}\\
		\bracket{\gamma}^{m-m_{i-1}-1} \bracket{\gamma w_1 \gamma^{-1}} \bracket{\gamma w_2 \gamma^{-1}} \dots \bracket{ \gamma w_j \gamma^{-1}} \bracket{&\gamma} \bracket{w_{j+1}} \dots\\
		\bracket{\gamma}^{m-(m_{i-1}+1)} \bracket{\gamma w_1 \gamma^{-1}} \bracket{\gamma w_2 \gamma^{-1}} \dots \bracket{ \gamma w_j \gamma^{-1}} \bracket{\gamma^{k+1} w^{-1} \gamma^{-k}} &\bracket{\gamma^k w \gamma^{-k}} \bracket{w_{j+1}} \dots~.
	\end{align*}
	
	This time we use \ref{eq:Cora} from Corollary~\ref{Cor:pos} for the last step.
	
	Now, if in sequence \ref{seq1} two words differ by a relation $\bracket{w_1} \bracket{w_2} = \bracket{w_3}$ for elements $w_1, w_2, w_3 \in \NC(W)$, then we will use the same relation in our new sequence. The only thing that might be different is that the elements could have been conjugated by $\gamma$ in previous steps. To show that also $\bracket{\gamma w_1 \gamma ^{-1}} \bracket{\gamma w_2 \gamma ^{-1}} \doteq \bracket{\gamma w_3 \gamma ^{-1}}$ holds (or an iteration of this equation), we have to show that this is also a defining relation, i.e. that $\gamma w_1 \gamma ^{-1} \, \gamma w_2 \gamma ^{-1} = \gamma w_3 \gamma ^{-1}$ and that $l(\gamma w_1 \gamma ^{-1}) + l(\gamma w_2 \gamma ^{-1}) = l(\gamma w_3	 \gamma ^{-1})$. The first equation is clear. For the second one note that reflection length is a conjugacy invariant.\\	

 	What remains to do is to define what we now do whenever $\bracket{w} \bracket{w}^{-1}$ or $\bracket{w}^{-1} \bracket{w}$ was deleted in the original sequence. For this, first note that whenever we introduced $\bracket{\gamma^{k+1} w \gamma^{-(k+1)}} \bracket{\gamma^{k+1} w^{-1} \gamma^{-k}}$ or $\bracket{\gamma^{k+1} w^{-1} \gamma^{-k}} \bracket{\gamma^k w \gamma^{-k}}$, afterwards there was $\bracket{\gamma^{k+1} w' \gamma^{-(k+1)}}$ left of it and $\bracket{\gamma^k w'' \gamma^{-k}}$ right of it with $w', w'' \in \NC(W)$ such that $\bracket{w'} \bracket{w} \bracket{w}^{-1} \bracket{w''}$ – or $\bracket{w'} \bracket{w}^{-1} \bracket{w} \bracket{w''}$ respectively – is an infix of the corresponding word in the original sequence.
 	
 	This is true since to the right of the newly introduced inverse there are $k$ more inverses, each of which accounts for one conjugation, and on the left there is one additional conjugation since we just inserted another inverse.
 	
 	With this observation we know that when $\bracket{w} \bracket{w}^{-1}$ or $\bracket{w}^{-1} \bracket{w}$ was deleted in the original sequence the two elements now correspond to $\bracket{\gamma^{k+1} w \gamma^{-(k+1)}} \bracket{\gamma^{k+1} w^{-1} \gamma^{-k}}$ or $\bracket{\gamma^{k+1} w^{-1} \gamma^{-k}} \bracket{\gamma^k w \gamma^{-k}}$, respectively. In the first case we can use \ref{eq:Corb}, in the second case we can use \ref{eq:Cora} from Corollary \ref{Cor:pos} to replace the infix by $\bracket{\gamma}$. Then we can move this copy of $\bracket{\gamma}$ with \ref{eq:Corc} all the way to the left. Since $m_i = m_{i-1}-1$ the resulting word has again the required form.
 	
 	All words in the constructed sequence are positive and we also saw that either a defining relation was used or we already knew that the two consecutive words were positively equal. This gives $\bracket{\gamma}^m X \doteq \bracket{\gamma}^m Y$.
 	
\end{proof}

The structure of this proof followed \cite{BRA01}. Birman, Ko and Lee have proven in \cite{BKL98} a similar embedding theorem (Theorem~2.7) on their way to give solutions to the word and conjugacy problems in the case of braid groups. In \cite{GAR69} Garside also gave a similar embedding theorem (Theorem~4) on his way to giving a solution to the conjugacy problem in braid groups.

\newpage


\section{Construction of the Classifying Space}\label{sec:complex}

For this section fix one arbitrary finite Coxeter group $W$ and let $\Gamma$ be the poset group related to $W$, as defined in Section~\ref{Sec:PosetGr}.

Since we require some basic knowledge about topology in general and covering theory in particular, the reader who is not familiar with this field may be referred to Jänich \cite{JAE05} or Bredon \cite{BRE93}.

\subsection{The Simplicial Complex X}\label{subsec:X}

We are now prepared to define a simplicial complex which is the universal cover of the $K(\Gamma,1)$. To prove this we will later have to show that it is contractible.

\begin{Definition}
	Let $X$ be the abstract simplicial complex defined as follows.
	The vertex set of $X$ is $\Gamma$ and a subset $f = \{ g_0, g_1, \dots g_k \} \subseteq \Gamma$ is a face of $X$ if $g_i = g_0 \bracket{w_i}$ for $i = 1, \dots k$ and $w_i \in \NC(W)$ with $I < w_1 < w_2 < \dots < w_k$.
\end{Definition}

That this is indeed an abstract simplicial complex shows the following lemma.

\begin{Lemma}
	$X$ is an abstract simplicial complex.
\end{Lemma}

\begin{proof}
	Let $F = \{ g_0, g_1, \dots g_k \}$ be a face of $X$ and $F' = \{ g_{i_0}, g_{i_1}, \dots, g_{i_l} \} \subseteq F$ a subset of this face. We only have to show that $F'$ is also a face of $X$.
	
	Since $F \in X$, we know that $g_i = g_0 \bracket{w_i}$ for $w_i \in \NC(W)$ with $I < w_1 < \dots < w_k$. Without loss of generality we can assume that $0 \leq i_0 < i_1 < \dots < i_l \leq k$. And thus, $ I < w_{i_1} < w_{i_2} < \dots < w_{i_l}$. Then there are two possible cases.
	
	Either $i_0 = 0$, then it is trivial to show $F' \in X$ since we have $g_{i_0}= g_0$ and $g_{i_j} = g_{i_0} \bracket{w_{i_j}}$ for $j=1, \dots, l$ and $w_{i_j} \in \NC(W)$ with $I < w_{i_1} < w_{i_2} < \dots w_{i_l}$.
	
	Or we have $i>0$, in which case $g_{i_0} = g_0 \bracket{w_{i_0}}$ for $w_{i_0} \in \NC(W) \setminus I$. But since $w_{i_0} \leq w_{i_j}$ for $j=1, \dots, l$, we know that $\bracket{w_{i_0}} \bracket{w_{i_0}^{-1} w_{i_j}} = \bracket{w_{i_j}}$ is a relation in $\Gamma$ and hence
	\[
		g_{i_j} = g_0 \bracket{w_{i_j}} = g_0 \bracket{w_{i_0}} \bracket{w_{i_0}^{-1} w_{i_j}} = g_{i_0} \bracket{w_{i_0}^{-1} w_{i_j}}.
	\]
	Lemma~\ref{Lem:3.10} gives $I < w_{i_0}^{-1} w_{i_1} < w_{i_0}^{-1} w_{i_2} < \dots < w_{i_0}^{-1} w_{i_l}$ and we have shown that $F'$ is a face of $X$.
\end{proof}

This complex can best be imagined as follows. Take the elements of $\Gamma$ as vertices and label each vertex with its corresponding element. Let $\Delta := \Delta\big( \NC(W) \big)$ be the order complex of $\NC(W)$.

Then glue to each vertex one copy of $\Delta$ – the identity in $\Delta$ onto the vertex of $\Gamma$ – , and label the vertices of the copy of $\Delta$ glued to $g$ by $g$ and the respective element of $\Delta$ in brackets. For example, consider the copy of $\Delta$ which was glued to the vertex labeled $g_0 \in \Gamma$. The vertex which corresponds to $w_1 \in \Delta$ is then labeled $g_0 \bracket{w_1}$.

Then, if for a vertex $g_0 \in \Gamma$ the element $g_0 \bracket{w_1}$ of the copy of $\Delta$ at $g_0$ is the same element as $g_1 \in \Gamma$, glue those vertices together.

This construction is of combinatorial nature. So whenever we talk about gluing, we actually mean identifying. This also implies that whenever we would have a double edge (or faces of higher dimension on the same vertex set) in the the geometric realization, we remove one edge (or face respectively).

\begin{Example}\label{ex:X}

	Consider the symmetric group on $3$ elements $S_3$ with the set of reflections $T= \big\{(1,2),(2,3),(1,3)\big\}$ and the Coxeter element $\gamma = (1,2,3)$. In this case we have $\Gamma = \Gamma(S_3, \gamma) \cong B_3$, the braid group on $3$ strands. We use Brady's notation from \cite{BRA01} and write $\bracket{i_1,  \dots, i_k}$ for the generator in $\Gamma$ which corresponds to $(i_1, \dots, i_k) \in \NC(S_3, \gamma)$.
	
	The maximal chains in $\NC(S_3, \gamma)$ are
	\[
		I < (1, 2) < (1, 2, 3), \quad
		I< (1, 3) < (1, 2, 3), \quad
		I < (2, 3) < (1, 2, 3).
	\]

		Thus, the geometric realization of the order complex $\Delta = \Delta\big(\NC(S_3, \gamma)\big)$ of the non-crossing partition lattice $\NC(S_3, \gamma)$ looks as displayed in Figure~\ref{fig:ord_comp}.
	
	\begin{figure}[!h]
		\centering
		\includegraphics[scale=0.7]{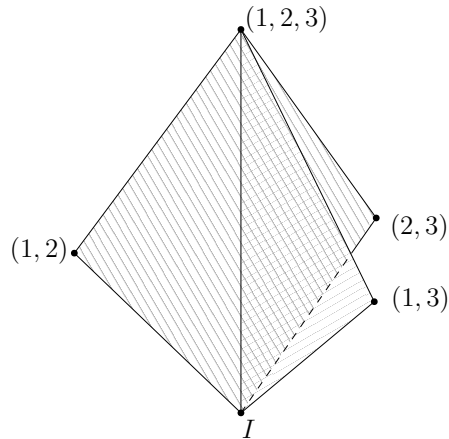}
		\caption{Order complex $\Delta(\NC(S_3, \gamma))$ of the symmetric group on $3$ elements}
		\label{fig:ord_comp}
	\end{figure}

	Then take for each element of $\Gamma = B_3$ one vertex and glue one copy of the order complex $\Delta$ to each of the vertices. Since $B_3$ has infinitely many elements, only a part of the geometric realization of this construction step is depicted in Figure~\ref{fig:Comp1}.
	
	\begin{figure}[!h]
		\centering
		\includegraphics[scale=0.68]{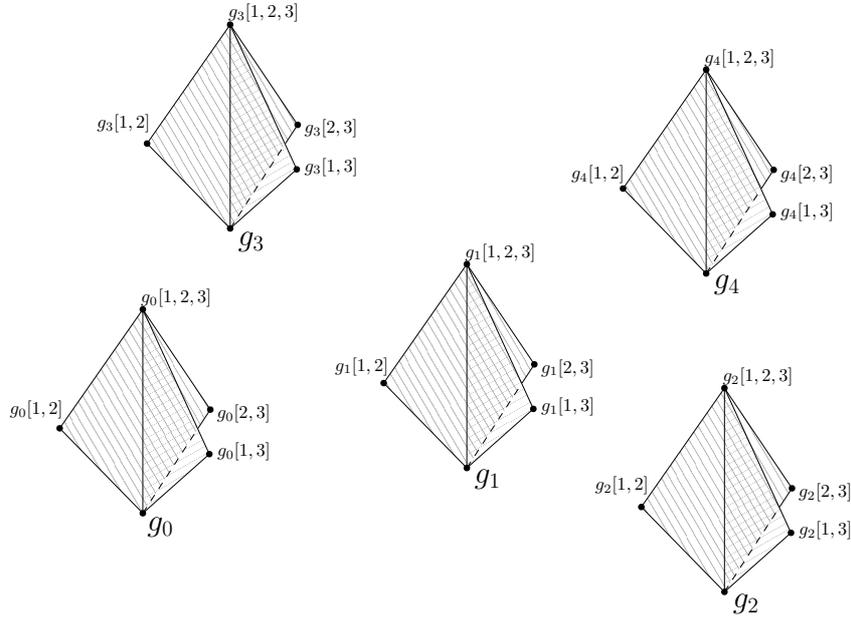}
		\caption{Gluing to each element of $B_3$ the order complex $\Delta$}
		\label{fig:Comp1}
	\end{figure}

Then finally, we glue vertices together, which are labeled by the same element from $\Gamma$. Figure~\ref{fig:Comp2} displays how the copy of $\Delta$ at the identity in $\Gamma$ and the one at the element $\bracket{1,2}$ are glued together. The two relations which are used to identify elements with one another are $I \bracket{1,2} = \bracket{1,2}$ and $\bracket{1,2} \bracket{2,3} = \bracket{1,2,3}$.

\begin{figure}[!h]
	\centering
	\includegraphics[scale=0.8]{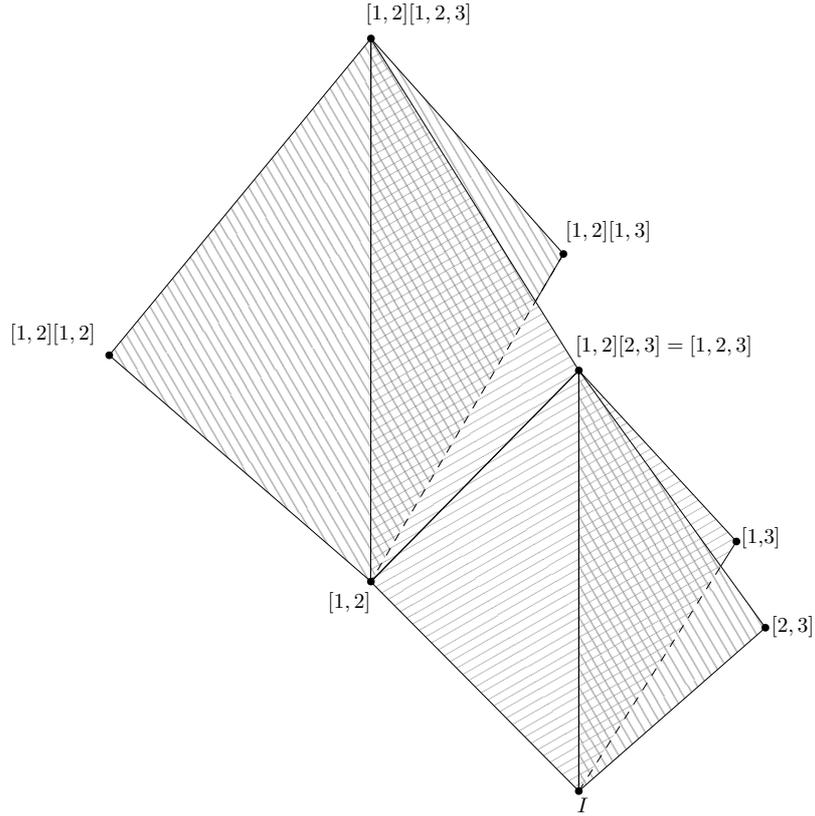}
	\caption{The copy of $\Delta$ at $I$ and $\bracket{1,2}$ are glued together}
	\label{fig:Comp2}
\end{figure}

\end{Example}

We can label not only the vertices of this complex but also the edges. For this, we label each edge – or $1$-dimensional face – by the element of $\NC(W)$ by which its two vertices differ. Consider for example the face $\{g_0,g_1\} \subseteq \Gamma$ with $g_1 = g_0 \bracket{w_1}$ and $\bracket{w_1} \in \NC(W)$. The corresponding edge is then labeled by $g_0^{-1}g_1 = \bracket{w_1}$.

Since the set $\NC(W)$ is by definition of $\Gamma$ a set of generators of $\Gamma$, $X$ must be connected. To see this, note that any vertex can be connected to the identity via the path that is labeled by a word in the generators (or inverses) of $\Gamma$ which represents the element.

\subsection{The Contractible Subcomplex}

The aim of this section is to show that $X$ is contractible. Having proven this fact we can later quickly derive that there is a $K(\Gamma,1)$. We will therefor first show that the following subcomplex of $X$ is contractible. The notion and arguments follow again Brady \cite{BRA01}.

\begin{Definition}
	Let $X^+$ be the subcomplex of $X$ which consists only of the faces of $X$ whose vertices can be labeled by positive words.
	
	Let $X_m^+$ be the subcomplex of $X^+$ which consists only of the faces of $X^+$ whose vertices can be labeled by positive words of length at most $m$.
\end{Definition}

Since any face of $X^+$ is already a face of $X$, any subset of it must also be in $X$. But since the vertices of the subset are also labeled by positive words, it is again a face of $X^+$. Thus, $X^+$ is closed under containment and therefore an abstract simplicial complex itself. For the same reason $X_m^+$ must be an abstract simplicial complex.

Before we can show that $X$ is contractible, we will first prove the following theorem.

\begin{Theorem}\label{thm:Xplus}
	The simplicical complex $X^+$ is contractible.
\end{Theorem}

For the proof of this theorem, we need a lemma and more definitions besides the known ones of \emph{closure, star} and \emph{link} (see Section~\ref{subsec:SK}).\\

\begin{Lemma}\label{lem:starfinite}

	For a vertex $v$ of $X$ the simplicial complex $\st(v,X)$ is finite.

\end{Lemma}

\begin{proof}

	Let $P$ be the label at $v$. Then all vertices in $\st(v,X)$ have labels of the form $P \bracket{ g }$ or $P \bracket{ g }^{-1}$, where $g \in \NC(W)$. Since $\NC(W)$ is finite, the lemma follows.

\end{proof}

\begin{Definition}
	Let $v$ be a vertex of $X$.
	\begin{enumerate}[label=(\roman*)]
		\item The \emph{ascending star} $\st_+(v,X)$ is defined as the closure of the union of faces of $X$ that are of the form
		$\big\{g, g \bracket{w_1}, g \bracket{w_2}, \dots, g \bracket{w_k}\big\}$ where $I < w_1 < w_2 < \dots < w_k$ is a chain in $\NC(W)$ and $g$ the label at $v$.
		
		\item The \emph{descending star} $\st_-(v,X)$ is defined as the closure of the union of faces of $X$ that are of the form
		$\big\{g, g \bracket{w_1}, g \bracket{w_2}, \dots, g \bracket{w_k}\big\}$ where $I < w_1 < w_2 < \dots < w_k$ is a chain in $\NC(W)$ and $g \bracket{w_k}$ the label at $v$.
		
		\item The \emph{ascending link} $\lk_+(v,X)$ is defined as the union of the faces of $\st_+(v,X)$, which do not have $v$ as a vertex.
		
		\item The \emph{descending link} $\lk_-(v,X)$ is defined as the union of the faces of $\st_-(v,X)$, which do not have $v$ as a vertex.
		
	\end{enumerate}
	
\end{Definition}

The same definitions can also be made with $X$ substituted by $X^+$. In this case we only consider faces of $X^+$.

Note that since the just defined complexes are subcomplexes of $\st(v,X)$, by Lemma~\ref{lem:starfinite} they are also finite.

\begin{proof}[Proof of Theorem~\ref{thm:Xplus}]

	We will show that for $m \in \N_0$, $X_{m+1}^+$ deformation retracts onto $X_{m}^+$. Having proven this, we know that $X^+$ is contractible since $X_0^+$ only consists of one vertex.
	
	So let $m \in \N_0$. To show that $X_{m+1}^+$ deformation retracts onto $X_{m}^+$, we first have to get an idea of what the difference between the two simplicial complexes is. Hereby $X_m^+$ is seen as a subset of $X_{m+1}^+$. Since the difference of the vertex sets only is all vertices which are labeled by a positive word of length $m+1$, we have to understand how these vertices are connected to $X_m^+$.
	
	Let $u$ be a vertex of $X_{m+1}^+ \setminus X_m^+$, i.e. it is labeled by a positive word of length $m+1$. Then, in $X_{m+1}^+$, $u$ is connected to $X_m^+$ via all faces of $X_{m+1}^+$ that have $u$ as a vertex. Because of the fact that all labels in $X^+$ are positive words and by the definition of a face in $X^+$, it is clear that the vertices of such a face must have strictly increasing length. Hence, in a face of $X_{m+1}^+$ that contains $u$, the vertex $u$ must be the only vertex labeled by a positive words of length $m+1$, and all other vertices in such a face have length strictly smaller than $m+1$. Thus, in $X_{m+1}^+$, $u$ is connected to $X_m^+$ exactly via all faces of $\st_-(u,X^+) \setminus \lk_-(u,X^+)$.
	
	For the same reason it is true that if $u$ and $w$ are vertices labeled by positive words of length $m+1$, $w$ can never be a vertex of $\st_-(u,X^+)$ and vice versa. Therefore, $\st_-(u,X^+) \cap \st_-(w,X^+) \subseteq X_m^+$.
	
	With this observation in mind, we only have to show for each vertex $u$ of $X_{m+1}^+ \setminus X_m^+$ that $\st_-(u,X^+)$ can be deformation retracted onto $\lk_-(u,X^+) \subseteq X_m^+$.\\
	
	To show this, note the following. Let $\tilde{X}_u := \st_-(u,X^+)$. Then $\tilde{X}_u$ is the smallest simplicial complex that contains all faces where the longest label is at the vertex $u$. Therefore, $\st(u,\tilde{X}_u) = \st_-(u,X^+)$ and $ \lk(u, \tilde{X}_u) = \lk_-(u, X^+)$. Thus, by Lemma~\ref{lem:cone} $\st(u,\tilde{X}_u) = \Cone_u(\lk(u,\tilde{X}_u))$; or in other words $\st_-(u,X^+) = \Cone_u(\lk_-(u, X^+))$. If we can now show that $\lk_-(u,X^+)$ is contractible, then it follows from Lemma~\ref{lem:deformation} that $\st_-(u,X^+) = \Cone_u(\lk_-(u, X^+))$ deformation retracts onto $\lk_-(u, X^+)$, which is what we want to show.
	
	So, what is left to show is that $\lk_-(u,X^+)$ is contractible. For this, let $P$ be the positive word labeling the vertex $u$, and $v_1, \dots, v_k$ all vertices in $\lk_-(u,X^+)$ with labels $Q_1, \dots, Q_k$, respectively. Note, that by Lemma~\ref{lem:starfinite} $\lk_-(u,X^+)$ is finite. Then, by definition of the descending link it follows that there are $x_1, \dots, x_k \in \NC(W)$ so that
	\[
		P = Q_1 \bracket{x_1} = \dots = Q_k \bracket{x_k}.
	\]
	Using Theorem~\ref{thm:embed}, we can deduce that $P \doteq Q_1 \bracket{x_1} \doteq \dots \doteq Q_k \bracket{x_k}$. Then we can apply Corollary~\ref{Cor:cancright} to obtain a positive word $Q$ with $P \doteq Q \bracket{x}$ for $x = x_1 \join \dots \join x_k \in \NC(W)$. But then there must be a vertex labeled with $Q$ – call it $v$ – which is also in $\lk_-(u,X^+)$.
	
	Then, for a vertex $v_i \in \lk_-(u, X^+)$ with label $Q_i$, we have $x_i \leq x$ and thus $x x_i^{-1} \leq x$ by Lemma~\ref{Lem:3.9}. And since we then have $Q \bracket{x x_i^{-1}} \bracket{x_i} \doteq Q \bracket{x} \doteq P \doteq Q_i \bracket{x_i} $ it follows that $Q \bracket{x x_i^{-1}} \doteq Q_i$. Hence, for any face $F$ of $\lk_-(u, X^+)$, the face $F \cup \{v\}$ is also in $\lk_-(u, X^+)$. This means that $\lk_-(u, X^+)$ is a cone with cone vertex $v$ and therefore contractible by Lemma~\ref{lem:conecontr}.
	
\end{proof}

\begin{Lemma}\label{lem:action}

	Let $g \in \Gamma$. Then there exists $k \in \N$ such that $\bracket{\gamma}^k g$ can be represented by a positive word.

\end{Lemma}

\begin{proof}

	Let $A$ be a word of the form $A = P_1 \bracket{r_1}^{-1} P_2 \bracket{r_2}^{-1} \dots \bracket{r_k}^{-1} P_{k+1}$, where $P_1, \dots, P_{k+1} \in \Gamma$ are positive words and $r_1, \dots, r_k \in \NC(W)$ such that $A$ represents $g \in \Gamma$. For any $g \in \Gamma$ we can find a word of this form which represents it. Then we can proceed as in the proof of Theorem~\ref{thm:embed}. We have
	
	\begin{align*}
	 \bracket{\gamma}^k A &= \bracket{\gamma}^k P_1 \bracket{r_1}^{-1} P_2 \bracket{r_2}^{-1} \dots\\
	 &\stackrel{\mathmakebox[\widthof{=}]{\text{\ref{Lem:pos} \ref{eq:Lemmac}}}}{=} \, \bracket{\gamma}^{k-1} Q_1 \bracket{\gamma} \bracket{r_1}^{-1} P_2 \bracket{r_2}^{-1} \dots\\
	 &\stackrel{\mathmakebox[\widthof{=}]{\text{\ref{Lem:pos} \ref{eq:Lemmaa}}}}{=} \, \bracket{\gamma}^{k-1} Q_1 \bracket{\gamma r_1^{-1}} \bracket{r_1} \bracket{r_1}^{-1} P_2 \bracket{r_2}^{-1} \dots\\
	 &= \bracket{\gamma}^{k-1} Q_1 \bracket{\gamma r_1^{-1}} P_2 \bracket{r_2}^{-1} \dots
	\end{align*}
	with $Q_1$ being the positive word obtained from $P_1$ by replacing each formal generator $\bracket{w}$ in $P_1$ by $\bracket{\gamma w \gamma^{-1}}$. In the last word the whole prefix $\bracket{\gamma}^{k-1} Q_1 \bracket{\gamma r_1^{-1}} P_2$ is a positive word and hence we have reduced the number of occurrences of inverses by one. Continuing like this we can inductively remove all inverses and obtain a positive word.
	
\end{proof}

\begin{Definition}

	Let $g, g_0, g_1 \dots g_k \in \Gamma$ and $\{g_0, \dots, g_k\}$ a face of $X$. Then $\Gamma$ acts on $X$ by
	\[
		g \cdot \{g_0, \dots, g_k\} := \{gg_0, \dots, gg_k\}
	\]

\end{Definition}

This action is simplicial, since if $\{g_0, \dots, g_k\}$ is a face of $X$, we have $g_i = g_0 \bracket{w_i}$ for $i = 1, \dots k$ and $w_i \in \NC(W)$ with $I < w_1 < w_2 < \dots < w_k$. It then follows immediately that $g g_i = g g_0 \bracket{w_i}$ for $i = 1, \dots k$ and $w_i \in \NC(W)$ with $I < w_1 < w_2 < \dots < w_k$.

\begin{Lemma}\label{lem:Xlimit}

	\[X = \lim_{k \to \infty} \bracket{\gamma}^{-k} X^+.\]

\end{Lemma}

\begin{proof}

	Since the action of $\Gamma$ on $X$ is simplicial, it follows that the right side is a subcomplex of the left side.
	
	For the other way around, let $v$ be a vertex in $X$. By Lemma~\ref{lem:action}, there exists $k_v \in \N$ such that $\bracket{\gamma}^{k_v} \cdot v$ is labeled by a positive word and therefore a vertex of $X^+$. Thus, it follows that $v$ is a vertex of $\bracket{\gamma}^{-k} X^+$ for all $k \geq k_v$ and therefore also of $\lim_{k \to \infty} \bracket{\gamma}^{-k} X^+$.
	
	To see that any face $f = \{g_0, \dots, g_k\}$ of $X$ is a face of the right side, consider $\bracket{\gamma}^k \cdot f$ for $k$ being the maximum number of appearances of inverses of formal generators among the $g_i$. Then, $\bracket{\gamma}^k \cdot f$ is a face of $X^+$ and we can use the same argument as before combined with the fact that the action is simplicial.
	
\end{proof}

\begin{Theorem}\label{thm:Xcontr}

	The simplicial complex $X$ is contractible.

\end{Theorem}

\begin{proof}

	We use Lemma~\ref{lem:Xlimit} and show that $\lim_{k \to \infty} \bracket{\gamma}^{-k} X^+$ is contractible.
	
	Since the simplicial action of $\Gamma$ on $X$ induces a continuous action on the geometric realization, the map $X \to X,\, x \mapsto g \cdot x$ induces for each $g \in \Gamma$ a homeomorphism on the geometric realization.
	
	Using this fact and the fact that $X^+$ is contractible by Theorem~\ref{thm:Xplus}, we can deduce that $\bracket{\gamma}^{-k} X^+$ is contractible for $k \geq 1$.
	
	Now, since $X$ is the direct limit of a sequence of inlcusions, the homotopy functor commutes with the direct limit by Lemma~\ref{lem:may} and we have
	\[
		\pi_n(X) \stackrel{\text{\ref{lem:Xlimit}}}{=} \pi_n(\lim_{k \to \infty} \bracket{\gamma}^{-k} X^+) \stackrel{\text{\ref{lem:may}}}{=} \lim_{k \to \infty} \pi_n(\bracket{\gamma}^{-k} X^+).
	\]
	Hence, all homotopy groups of $X$ are trivial.
	
	To see that $X$ is contractible, we now consider the map $f \colon X \to \{*\}$ that sends all points of $X$ to the one-point-space. Since all homotopy groups of $X$ are trivial, $f$ induces isomorphisms on the homotopy groups. Also, we can regard both $X$ and the one-point-space as connected CW-complexes. Hence we can apply Whitehead's Theorem~\ref{thm:Whitehead} to deduce that $f$ is a homotopy equivalence. Thus, $X$ is contractible.

\end{proof}

We now consider the following space.

\begin{Definition}
	
	Let $K :=\, \mfaktor{\Gamma}{X}$ be the quotient space of the geometric realization of $X$ under the action of $\Gamma$.

\end{Definition}

Since the action of $\Gamma$ on $X$ is transitive on the vertices, $K$ consists of only one vertex. Recall the construction of $X$ in Section~\ref{subsec:X}. All copies of the order complex of $\NC(W)$ which were glued to the vertices corresponding to elements in $\Gamma$ are identified with one another under the quotient map $X \to K = \, \mfaktor{\Gamma}{X}$. Therefore, $K$ consists of one vertex and one copy of the order complex of $\NC(W)$, of which each vertex is identified with the single vertex in $K$. Additionally, the other faces of the order complex are identified via
\[
	\{w_0, w_1, \dots, w_k\} \sim \{I, w_0^{-1}w_1, \dots, w_0^{-1}w_k\}.
\]
The quotient space $K$ then forms a finite CW-complex.

\begin{Example}\label{ex:K}
	Continuing Example~\ref{ex:X}, we consider the symmetric group on $3$ elements and the braid group on $3$ strands, respectively, with Coxeter element $\gamma=(1,2,3)$.
	
	The faces of $\Delta\big(\NC(S_3,\gamma)\big)$ which are identified with one another are
	\begin{align*}
		\{I\} \sim \big\{(1,2)\big\} \sim &\big\{(2,3)\big\} \sim \big\{(1,3)\big\} \sim \big\{(1,2,3)\big\},\\
		\big\{I, (1,2)\big\} &\sim \big\{(1,3), (1,2,3)\big\},\\
		\big\{I, (2,3)\big\} &\sim \big\{(1,2), (1,2,3)\big\},\\
		\big\{I, (1,3)\big\} &\sim \big\{(2,3), (1,2,3)\big\}.
	\end{align*}
	 
	 Besides the obvious identification of all vertices with one another, Figure~\ref{fig:K} shows which other faces are identified.
	\begin{figure}[!h]
		\centering
		\includegraphics[scale=0.8]{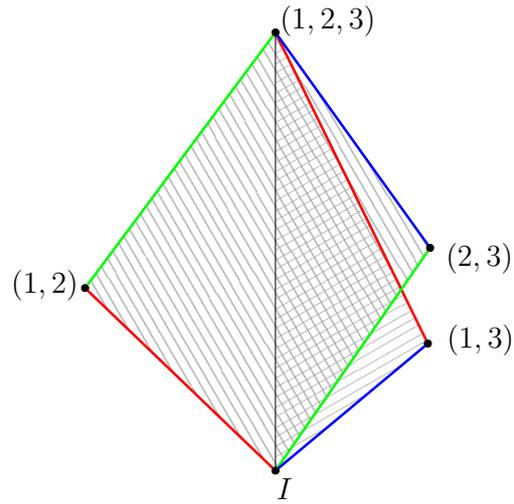}
		\caption{Building the $K(B_3,1)$ from the order complex of $\NC(S_3,\gamma)$}
		\label{fig:K}
	\end{figure}

\end{Example}

\begin{Theorem}
	
	The CW-complex $K$ is a $K\big(A(W),1 \big)$.

\end{Theorem}

\begin{proof}

	Since the deck transformations of the covering map $X \to K$ are exactly the maps $X \to X, \, x \mapsto g \cdot x$ for all $g \in \Gamma$, the deck transformation group is isomorphic to $\Gamma$, which is in turn isomorphic to $A(W)$ by Theorem~\ref{thm:iso}. We also know that $X$ is the universal covering space since by Theorem~\ref{thm:Xcontr} the complex $X$ is contractible and hence simply connected. Thus, the fundamental group of $K$ is isomorphic to the deck transformation group,
	\[
		\pi_1(K) \cong A(W).
	\]
	
	To see that all $\pi_n\big(A(W)\big)$ are trivial for $n \geq 2$, we consider Theorem~\ref{thm:pi_two}. As noted after the theorem, the map $\pi_n(X) \to \pi_n(K)$ is an isomorphism for $n \geq 2$, since the quotient map $X \to K$ is a covering map.
	
	This shows that $K$ is a $K\big(A(W),1 \big)$.

\end{proof}

 \newpage  

\clearpage
  \newpage

	\bibliography{Thesis}

\clearpage

\begin{appendix}
\appendixpage
\addappheadtotoc
\section{Topological facts}

	This appendix is a collection of topological facts that are used in the thesis but are thematically detached from the rest.\\

	In the following lemmas a more general definition of cone is used. A \emph{cone} over a topological space $Z$ is defined to be $\Cone(Z) := \big(Z \times [0,1] \big) / \big(Z \times \{1\}\big)$. We are not going into further detail but note that in the case of simplicial complexes the definition coincides with our Definition~\ref{def:SK} up to homeomorphism.

\begin{Lemma}\label{lem:conecontr}
	Let $Z$ be a topological space. Then $\Cone (Z)$ is contractible.
\end{Lemma}

For a proof, see for example Theorem~1.11 in Rotman's \cite{ROT88}.

\begin{Lemma}\label{lem:deformation}

	Let $Z$ be a contractible space. Then $\Cone(Z)$ deformation retracts onto $Z$.

\end{Lemma}

Since no proof of this fact could be found in the literature, we give a short proof here.

\begin{proof}

	We denote a point of $\Cone(Z)$ as $[z,s]$ with $z \in Z, s \in [0,1]$ if it is the image of $(z,s) \in Z \times [0,1]$ under the quotient map $Z \times [0,1] \to \big(Z \times [0,1] \big) / \big(Z \times \{1\} \big) = \Cone(Z)$. In particular $[z_1,1] = [z_2,1] \in \Cone (Z)$ for all $z_1, z_2 \in Z$.

	Since $Z$ is contractible, there exists a deformation retraction $H' \colon Z \times [0,1] \to Z$ of $Z$ onto a point $z_0 \in Z$. Hence we have $H'(z,0) = z$ and $H'(z,1) = z_0$ for all $z \in Z$ with $H'$ continuous.
	
	Then, a deformation retraction of $\Cone (Z) = \big(Z \times [0,1] \big) / \big(Z \times \{1\} \big)$ onto $Z \cong Z \times \{0\}$ is given by
	\[
		H([z,s],t) =
		\begin{cases}
			[H'(z,2st),s] & \text{for } 0 \leq t \leq 1/2 \\
			[H'(z,s), s (2-2t)] & \text{for } 1/2 \leq t \leq 1. \\
		\end{cases}
	\]
	
	To see that the map is well-defined, we have to show that at the cone point, i.e. when $s=1$, the definition does not depend on $z$. For $0 \leq t \leq 1/2$ and $s=1$, the right side is $[H'(z,2t),1]$, which does not depend on $z$ as noted above. For $1/2 \leq t \leq 1$ and $s=1$, the right side is $[H'(z,1), 2-2t]= [z_0, 2-2t]$, which as well does not depend on $z$. Since at $t= 1/2$ we have $[H'(z,2st),s] = [H'(z,s), s] = [H'(z,s), s (2-2t)]$, the map is well-defined and continuous.
	
	What is left to show is that $H$ is a deformation retraction.
	
	For this, note that
	\begin{align*}
		H([z,s],0) &= [H'(z,0),s] = [z,s] \\
		H([z,s],1) &= [H'(z,s), 0] \in Z \times \{0\} \\
		H([z,0],1) &= [H'(z,0), 0] = [z,0].
	\end{align*}
		
\end{proof}

The following theorem was first proven by Whitehead in \cite{WHI49}.

\begin{Theorem}[Whitehead's Theorem, \cite{HAT02}, Theorem~4.5]\label{thm:Whitehead}

	If a map $f \colon X \to Y$ between connected CW complexes induces isomorphisms $f_* \colon \pi_n(X)\to \pi_n(Y)$ for all $n$, then $f$ is a homotopy equivalence.
\end{Theorem}

For a proof and further reading on Whitehead's Theorem see for example Chapter~4.1 in Hatcher's \cite{HAT02} or Corollary~11.14 of Chapter~VII in Bredon's \cite{BRE93}.

\begin{Lemma}[\cite{MAY99}, Chapter~9.4]\label{lem:may}
	If $Z$ is the colimit of a sequence of inclusions $Z_i \to Z_{i+1}$ of based spaces, then the natural map $\colim_i \pi_n(Z_i) \to \pi_n(Z)$ is an isomorphism for each $n$.
\end{Lemma}

For a proof of this, see May \cite{MAY99}, Chapter~9.4.

\begin{Theorem}[\cite{BRE93}, Chapter~VII, Theorem~6.7]\label{thm:pi_two}

	If $p\colon Y \to B$ is a fibration and if $y_0 \in Y, b_0= p (y_0)$, and $F=p^{-1} (b_0)$, then taking $y_0$ as the base point of $Y$ and of $F$ and $b_0$ as the base point of $B$, we have the exact sequence:
	
	\begin{align*}
		&\cdots \longrightarrow \pi_n(F) \stackrel{i_{\#}}{\longrightarrow} \pi_n(Y) \stackrel{p_{\#}}{\longrightarrow} \pi_n(B) \stackrel{\partial_{\#}}{\longrightarrow} \pi_{n-1}(F) \longrightarrow \cdots\\
		&\cdots \longrightarrow \pi_1(Y) \longrightarrow \pi_1(B) \longrightarrow \pi_0(F) \longrightarrow \pi_0(Y) \longrightarrow \pi_0(B).
	\end{align*}

\end{Theorem}

After the theorem Bredon notes that ``a covering map is clearly a fibration. In that case the fiber $F$ is discrete and so $\pi_n(F)= 0$ for $n \geq 1$. Thus the exact sequence implies that $p_{\#} \colon \pi_n(Y)\to \pi_n(B)$ is an isomorphism for $n \geq 2$'' \cite{BRE93}.

\end{appendix}

\end{document}